\documentclass[preprint,times]{elsarticle}

\usepackage[english]{babel}

\usepackage{amsmath,amssymb,amsthm}
\allowdisplaybreaks
\usepackage{graphicx}
\graphicspath{{figures/}}
\usepackage{hyperref}
\usepackage{bm}
\usepackage[hang,small,bf]{caption}
\usepackage{subcaption}
\usepackage{tabularx}
\usepackage{setspace}
\DeclareMathOperator{\interior}{int}
\DeclareMathOperator{\cl}{cl}
\DeclareMathOperator{\dom}{dom}

\DeclareMathOperator{\diag}{diag}

\DeclareMathOperator{\real}{\mathbb{R}}
\DeclareMathOperator{\comp}{\mathbb{C}}
\DeclareMathOperator{\re}{Re}

\DeclareMathOperator*{\argmin}{argmin}
\DeclareMathOperator{\cossim}{sim}

\newcommand{\ie}{\textit{i.e.}}

\newtheorem{theorem}{Theorem}
\newtheorem{corollary}{Corollary}

\newtheorem{definition}{Definition}
\newtheorem{remark}{Remark}

\usepackage[noend]{algpseudocode}

\usepackage{algorithm}

\algnewcommand{\Input}[1]{
\State\textbf{Input:}\hspace*{0.5em}\parbox[t]{.85\linewidth}{\raggedright #1}
}
\algnewcommand{\Initialization}[1]{
\State\textbf{Initialization:}\hspace*{0.5em}\parbox[t]{.74\linewidth}{\raggedright #1}
}

\journal{ }

\begin{document}

\begin{frontmatter}
    \title{Blind Deconvolution with Non-smooth Regularization via Bregman Proximal DCAs}
\author[inst1]{Shota~Takahashi}
\ead{takahashi.shota@ism.ac.jp}
\address[inst1]{
Department of Statistical Science, The Graduate University for Advanced Studies,
10--3 Midori-cho, 
Tachikawa,190--8562, 
Tokyo,
Japan}

\author[inst1,inst2,inst3]{Mirai~Tanaka}
\ead{mirai@ism.ac.jp}
\author[inst1,inst2]{Shiro~Ikeda}
\ead{shiro@ism.ac.jp}
\address[inst2]{
Department of Statistical Inference and Mathematics, The Institute of Statistical Mathematics
10--3 Midori-cho, 
Tachikawa,
190--8562, 
Tokyo,
Japan}
\address[inst3]{
Continuous Optimization Team, RIKEN Center for Advanced Intelligence Project,
Nihonbashi 1-chome Mitsui Building, 15th floor, 1--4--1 Nihonbashi, 
Chuo-ku,
103--0027, 
Tokyo,
Japan}

\begin{abstract}
Blind deconvolution is a technique to recover an original signal without knowing a convolving filter.
It is naturally formulated as a minimization of a quartic objective function under some assumption.
Because its differentiable part does not have a Lipschitz continuous gradient, existing first-order methods are not theoretically supported.
In this paper, we employ the Bregman-based proximal methods, whose convergence is theoretically guaranteed under the $L$-smooth adaptable ($L$-smad) property.
We first reformulate the objective function as a difference of convex (DC) functions and apply the Bregman proximal DC algorithm (BPDCA).
This DC decomposition satisfies the $L$-smad property.
The method is extended to the BPDCA with extrapolation (BPDCAe) for faster convergence.
When our regularizer has a sufficiently simple structure, each iteration is solved in a closed-form expression, and thus our algorithms solve large-scale problems efficiently.
We also provide the stability analysis of the equilibriums and demonstrate the proposed methods through numerical experiments on image deblurring.
The results show that BPDCAe successfully recovered the original image and outperformed other existing algorithms.
\end{abstract}
\begin{keyword}
Blind deconvolution \sep DC optimization \sep Bregman proximal DC algorithms \sep image deblurring
\end{keyword}
\end{frontmatter}

\section{Introduction}\label{sec:introduction}
We consider the convolution of a filter $\bm{f}\in\real^m$ and a signal $\bm{g}\in\real^m$, given by
\begin{align}
    \label{eq:blind-deconv}
    \tilde{\bm{y}} = \bm{f}*\bm{g}.
\end{align}
Our goal is to recover $\bm{g}$ from $\tilde{\bm{y}}$ without knowing $\bm{f}$.
This problem is known as blind deconvolution.
Blind deconvolution arises in many fields of science and engineering, such as sensor networks~\cite{balzano07}, optics~\cite{chen13}, astronomy~\cite{fetick20,jefferies93}, communication engineering~\cite{liu95}, and image processing~\cite{ahmed14,jain22,li19,zhao16}.

Without any assumptions, blind deconvolution is ill-posed.
A common approach is to assume that $\bm{f}$ and $\bm{g}$ belong to known subspaces~\cite{ahmed14}.
Concretely, for known linear operators $\tilde{\bm{B}}:\real^{d_1}\to\real^m$ and $\tilde{\bm{A}}:\real^{d_2}\to\real^m$, we assume that there exist the true vectors $\bm{h}^{\circ}\in\real^{d_1}$ and $\bm{x}^{\circ}\in\real^{d_2}$ such that
$\bm{f} = \tilde{\bm{B}}\bm{h}^{\circ}$ and $\bm{g} = \tilde{\bm{A}}\bm{x}^{\circ}$, where $d_1, d_2 < m$.
Moreover, we consider blind deconvolution in the Fourier domain.
Applying the DFT to both sides of~\eqref{eq:blind-deconv} and letting $\bm{F}\in\comp^{m\times m}$ be the unitary discrete Fourier transform (DFT) matrix, we obtain
\begin{align*}
    \sqrt{m}\bm{F}\tilde{\bm{y}} = \sqrt{m}\bm{F} (\bm{f} *\bm{g}) = \sqrt{m}\bm{F}\bm{f} \odot \sqrt{m}\bm{F}\bm{g} = m\bm{F}\tilde{\bm{B}}\bm{h}^{\circ} \odot \bm{F}\tilde{\bm{A}}\bm{x}^{\circ},
\end{align*}
where the second equality holds from the convolution theorem~\cite[Section 4.4.2]{proakis06} and $\odot$ denotes the Hadamard (elementwise) product.
Thus, \eqref{eq:blind-deconv} is rewritten in the Fourier domain as $\bm{y} = \bm{Bh}^{\circ} \odot \overline{\bm{Ax}^{\circ}}$, where $\bm{y} := \frac{1}{\sqrt{m}}\bm{F}\tilde{\bm{y}}$, $\bm{B} := \bm{F}\tilde{\bm{B}}$, and $\overline{\bm{A}} := \bm{F}\tilde{\bm{A}}$.

Now, the goal of our problem is to estimate $\bm{h}^{\circ}$ and $\bm{x}^{\circ}$ from $\bm{y}$.
To evaluate the fidelity, we consider the squared error function $F(\bm{h}, \bm{x}) = \frac{1}{2} \|\bm{Bh}\odot\overline{\bm{Ax}} - \bm{y}\|^2_2$.
In addition, in order to incorporate with the image characteristics, we also consider a regularization term $G:\real^{d_1+d_2}\to(-\infty, +\infty]$.
It may not be differentiable.
Commonly used regularization terms include non-differentiable functions, such as the $\ell_1$ norm and the total variation.
To compute $\bm{h}^{\circ}$ and $\bm{x}^{\circ}$, in this paper, we minimize the sum of these two functions and a constraint set $\cl{C}$, where $\cl{C}$ denotes the closure of a nonempty open convex set $C\subset\real^{d_1+d_2}$.
That is, we consider the following optimization problem:
\begin{align}
    \label{prob:blind-deconv}
    \min_{(\bm{h}, \bm{x})\in\cl{C}}\quad
    F(\bm{h}, \bm{x}) + G(\bm{h}, \bm{x}).
\end{align}
Note that $F$ is a quartic function because it has a quartic term $h_i h_j x_k x_l$ with respect to $(\bm{h}, \bm{x})$, and thus it does not have a Lipschitz continuous gradient.
Hence, we cannot rely on the convergence analysis of existing first order methods, such as the fast iterative shrinkage-thresholding algorithm (FISTA)~\cite{beck09}, because their convergence analysis depends on the existence of Lipschitz continuous gradients.
Instead, we try to resort to Bregman-based proximal gradient algorithms~\cite{beck18,wu21}.
These algorithms generalize the proximal gradient method by replacing the squared Euclidean distance with the Bregman distance $D_H$ associated with a kernel generating distance $H$.
The algorithms generates a sequence converging to a stationary point under the $L$-smooth adaptable ($L$-smad) property of $(F, H)$ defined later.
For our problem, however, finding an appropriate $H$ is difficult because of the bilinear term of $F$.

To address the problem, we first reformulate $F$ into a difference of convex (DC) functions and apply the Bregman proximal DC algorithm (BPDCA)~\cite{takahashi21}.
For our DC decomposition $F = F_1 - F_2$, we find an $H$ that makes $(F_1, H)$ $L$-smad, and hence BPDCA with corresponding $D_H$ converges to a stationary point of~\eqref{prob:blind-deconv}.
We also apply a variant, the BPDCA with extrapolation (BPDCAe)~\cite{takahashi21}, which converges faster.
When $G$ has a simple structure, each iteration of BPDCA(e) is solved in a closed form.
This fact implies that our algorithms solve large-scale problems efficiently.
We also show the stability property of the algorithms around the equilibriums and compare them with existing algorithms.
Moreover, we show a numerical example of image deblurring.

This paper is organized as follows. 
Section~2 shows that $F$ is reformulated into a DC functions, introduces the Bregman Proximal DC Algorithms, and provides the stability analysis.
Section~3 illustrates numerical experiments on image deblurring. 
Section~4 summarizes our contributions.

\paragraph*{Notation}
Vectors and matrices are shown in boldface.
All one vector is $\bm{1}\in \real^m$ and the identity matrix is $\bm{I}_d\in \real^{d\times d}$.
Let $|\bm{z}|$ and $\bm{z}^2$ be elementwise absolute and squared vectors for $\bm{z}\in\comp^d$, respectively.
$\re(\bm{z})$, $\bar{\bm{z}}$, and $\bm{z}^{\mathsf{H}}$ denote its real part, complex conjugate, and complex conjugate transpose, respectively.
The inner product of $\bm{z}, \bm{w}\in\comp^d$ is defined by $\langle\bm{z}, \bm{w}\rangle = \bm{z}^{\mathsf{H}} \bm{w}$.

In what follows, we use the following notations.
Let $\interior{C}$ be the interior of a set $C\subset\real^d$.
For a proper and lower semicontinuous function $F:\real^d\to(-\infty, +\infty]$, the domain $\{\bm{z}\in\real^d\ |\ F(\bm{z}) < +\infty\}$ of $F$ is denoted by $\dom F$, and the domain $\{\bm{z}\in\real^d\ |\ \partial F(\bm{z})\neq \emptyset\}$ of the limiting subdifferential of $F$ is also denoted by $\dom\partial F$, where $\partial F$ denotes the limiting subdifferential of $F$ defined in~\cite[Definition 5]{takahashi21}.
Note that the limiting subdifferential coincides with the (classical) subdifferential of $F$ under the convexity.

\section{Proposed Method}
\subsection{DC Decomposition}\label{sec:dc-optimization}
We first reformulate $F$ in~\eqref{prob:blind-deconv} into a DC function. Let us define $F_1$ and $F_2$ as follows:
\begin{align*}
    F_1(\bm{h}, \bm{x}) &= \frac{1}{4}\|\bm{Bh}\|^4_4 + \frac{1}{4}\|\bm{Ax}\|_4^4 + \frac{1}{2}\left(\|\bm{Bh}\odot\bm{Ax}\|^2_2 +\|\bm{y}\odot\bm{Bh}\|^2_2 + \|\bm{Ax}\|^2_2 + \|\bm{y}\|^2_2\right),\\
    F_2(\bm{h}, \bm{x}) &= \frac{1}{4}\|\bm{Bh}\|^4_4 + \frac{1}{4}\|\bm{Ax}\|_4^4 + \frac{1}{2}\|\bar{\bm{y}}\odot \bm{Bh} + \overline{\bm{Ax}}\|^2_2.
\end{align*}
$F_2$ is convex because a composite function of a linear transform and a convex function is convex~\cite[Theorem 3.1.6]{nesterov18}.
By denoting $\bm{b}_j$ and $\bm{a}_j$ be the $j$-th column vectors of $\bm{B}^{\mathsf{H}}$ and $\bm{A}^{\mathsf{H}}$, respectively, we have
\begin{align*}
    F_1(\bm{h}, \bm{x}) &= \frac{1}{4}\sum_{j=1}^m\left(|\bm{b}_j^{\mathsf{H}}\bm{h}|^2 + |\bm{a}_j^{\mathsf{H}}\bm{x}|^2\right)^2+\frac{1}{2}\left(\|\bm{y}\odot\bm{Bh}\|^2_2 + \|\bm{Ax}\|^2_2 + \|\bm{y}\|^2_2\right),
\end{align*}
which proves the convexity of $F_1$.
Since $F = F_1 - F_2$ holds,~\eqref{prob:blind-deconv} is equivalent to the following DC optimization problem:
\begin{align}
    \label{prob:blind-deconv-dc}
    \min_{(\bm{h}, \bm{x})\in\cl{C}}\quad
    F_1(\bm{h}, \bm{x}) - F_2(\bm{h}, \bm{x}) + G(\bm{h}, \bm{x}).
\end{align}

\subsection{Bregman Proximal DC Algorithms}
Before showing DC algorithms,
we define the kernel generating distances, the Bregman distances, and the $L$-smooth adaptable property.
\begin{definition}[Kernel generating distances~\cite{beck18} and Bregman distances~\cite{bregman67}]
    \label{def:kernel}
    Let $C$ be a nonempty open convex subset of $\real^d$.
    $H:\real^d\to(-\infty, +\infty]$ is called a kernel generating distance associated with $C$ if it meets the following conditions:
    \begin{enumerate}
        \renewcommand{\labelenumi}{\rm{(\roman{enumi})}}
        \item $H$ is proper, lower semicontinuous, and convex, with $\dom H \subset \cl C$ and $\dom \partial H = C$.
        \item $H$ is $\mathcal{C}^1$ on $\interior\dom H = C$.
    \end{enumerate}
    We denote the class of the kernel generating distances associated with $C$ by $\mathcal{G}(C)$.
    Given $H \in \mathcal{G}(C)$, the Bregman distance $D_H: \dom H \times \interior\dom H \to \real_+$ is defined by 
    \begin{align*}
        D_H(\bm{z}, \bm{w}) := H(\bm{z}) - H(\bm{w}) - \langle \nabla H(\bm{w}), \bm{z} - \bm{w} \rangle.
    \end{align*}
\end{definition}
\begin{definition}[$L$-smooth adaptable ($L$-smad)~\cite{beck18}]
    \label{def:l-smad}
    Consider a pair of functions $(F, H)$ satisfying the following conditions:
    \begin{enumerate}
        \renewcommand{\labelenumi}{\rm{(\roman{enumi})}}
        \item $H \in \mathcal{G}(C)$,
        \item $F: \real^d \to (-\infty, +\infty]$ is proper, lower semicontinuous, and $\mathcal{C}^1$ on $C = \interior\dom H$ with $\dom F \supset \dom H$.
    \end{enumerate}
    The pair $(F, H)$ is called $L$-smooth adaptable ($L$-smad) on $C$ if there exists $L > 0$ such that $L H - F$ and $L H + F$ are convex on $C$.
\end{definition}

\begin{algorithm}[H]
	\caption{BPDCA~\cite{takahashi21} for DC optimization problem~\eqref{prob:blind-deconv-dc}}
	\label{alg:bpdca}
    \begin{algorithmic}[t]
    \Input{$H\in\mathcal{G}(C)$, $\bm{z}^0 \in C$ and $0 < \lambda < 1 / L$.}
    \For{$k = 0, 1, 2, \ldots,$}
    \State Compute
    \begin{align}
        \label{subprob:bpdca}
        \bm{z}^{k+1} = \argmin_{\bm{z}\in \cl{C}}\{
        &\langle\nabla F_1(\bm{z}^k) - \nabla F_2(\bm{z}^k), \bm{z}\rangle + G(\bm{z}) + \frac{1}{\lambda}D_H(\bm{z}, \bm{z}^k)\}.
    \end{align}
    \EndFor
    \end{algorithmic}
\end{algorithm}

\begin{algorithm}[H]
	\caption{BPDCAe~\cite{takahashi21} for DC optimization problem~\eqref{prob:blind-deconv-dc}}
	\label{alg:bpdcae}
    \begin{algorithmic}[t]
    \Input{$H\in\mathcal{G}(C)$, $\bm{z}^{-1} = \bm{z}^0 \in C$, $t_{-1}=t_{0}=1$, $0 < \lambda < 1 / L$, $\rho\in[0, 1)$, and $N\in\mathbb{N}$.}
    \For{$k = 0, 1, 2, \ldots,$}
    \State Set $\bm{w}^k = \bm{z}^k + \frac{t_{k-1} - 1}{t_k}(\bm{z}^k - \bm{z}^{k-1})$, $t_{k+1} = \frac{1 + \sqrt{1 + 4t_k^2}}{2}$.
    \If{$k = k' N$ with $k'\in\mathbb{N}$, $\bm{w}^k \notin C$, or $D_H(\bm{z}^k, \bm{w}^k) > \rho D_H(\bm{z}^{k-1}, \bm{z}^k)$}
        \State Set $t_{k-1} = t_k = 1$ and $\bm{w}^k = \bm{z}^k$.
    \EndIf
    \State Compute
        \begin{align}
        \label{subprob:bpdcae}
            \bm{z}^{k+1}
            = \argmin_{\bm{z}\in \cl{C}}\{
            &\langle\nabla F_1(\bm{w}^k) - \nabla F_2(\bm{z}^k), \bm{z}\rangle + G(\bm{z}) + \frac{1}{\lambda}D_H(\bm{z}, \bm{w}^k)\}.
        \end{align}
    \EndFor
    \end{algorithmic}
\end{algorithm}

To solve~\eqref{prob:blind-deconv-dc}, we introduce the Bregman proximal DC algorithm (BPDCA)~\cite{takahashi21} in Algorithm~\ref{alg:bpdca}, where we denoted $\bm{z} = (\bm{h}, \bm{x})$.
At each iteration of Algorithm~\ref{alg:bpdca}, we minimize $\langle\nabla F_1(\bm{z}^k) - \nabla F_2(\bm{z}^k), \bm{z}\rangle + G(\bm{z}) + \frac{1}{\lambda}D_H(\bm{z}, \bm{z}^k)$.
The first and second terms of this objective function are a first-order approximation of the original objective function, that is, $F_1(\bm{z}) - F_2(\bm{z}) + G(\bm{z})\simeq F_1(\bm{z}^k) - F_2(\bm{z}^k) + \langle\nabla F_1(\bm{z}^k) - \nabla F_2(\bm{z}^k), \bm{z}\rangle + G(\bm{z})$.
The third term is the Bregman proximity $\frac{1}{\lambda}D_H(\bm{z}, \bm{z}^k)$ between $\bm{z}$ and $\bm{z}^k\in\cl{C}$.
The parameter $\lambda$ controls the balance of these two components.
We also introduce an accelerated version of BPDCA with extrapolation (BPDCAe)~\cite{takahashi21} in Algorithm~\ref{alg:bpdcae}.
We solve Subproblem~\eqref{subprob:bpdcae}, which is similar to Subproblem~\eqref{subprob:bpdca} in Algorithm~\ref{alg:bpdca}, at each iteration. 
See~\cite{takahashi21} for more details.
The function $G$ is possibly non-smooth, and when $G$ has a sufficiently simple structure,~\eqref{subprob:bpdca} and~\eqref{subprob:bpdcae} are easily solved.
For instance, we solve Subproblem~\eqref{subprob:bpdca} to obtain the sparse signal and filter when $G(\bm{h}, \bm{x}) = \theta_1\|\bm{h}\|_1 + \theta_2\|\bm{x}\|_1$ for $\theta_1, \theta_2 \geq 0$.
Let $\bm{u}=\mathcal{S}_{\lambda\theta_1}( \lambda\nabla_{\bm{h}} F(\bm{z}^k) - \nabla_{\bm{h}} H(\bm{z}^k))$ and $ \bm{v}=\mathcal{S}_{\lambda\theta_2}( \lambda\nabla_{\bm{x}} F(\bm{z}^k) - \nabla_{\bm{x}} H(\bm{z}^k))$, where $\mathcal{S}_{\tau}$ is the soft-thresholding operator~\cite{beck18} with $\tau>0$.
We can prove from~\cite[Proposition 5.1]{beck18} that $\bm{z}^{k+1} = (-t^*\bm{u}, -t^*\bm{v})$ solves Subproblem~\eqref{subprob:bpdca}, where $t^*$ is the unique positive real root of the cubic equation $t^3(\|\bm{u}\|^2_2+\|\bm{v}\|^2_2)+t-1=0$.
Note that every cubic equation has a closed-form solution via Cardano's formula.
This fact indicates solution of~\eqref{subprob:bpdca} is expressed in closed form.
It is also true for Subproblem~\eqref{subprob:bpdcae}.
When $(F_1, H)$ is $L$-smad, the convergence of the algorithms to a stationary point is guaranteed~\cite{takahashi21}.
In theory and practice, BPDCAe converges faster than BPDCA~\cite{takahashi21}, whereas BPDCAe requires the convexity of $G$.

The following theorem provides appropriate $H$ and $L$ in use of Algorithms~\ref{alg:bpdca} and~\ref{alg:bpdcae}.
\begin{theorem}
\label{theo:l-smad-para}
    Let $H$ be defined by 
    \begin{align}
        \label{def:kernel-blind-dc}
        H(\bm{h}, \bm{x}) = \frac{1}{4}\left(\|\bm{h}\|^2_2 + \|\bm{x}\|^2_2\right)^2 + \frac{1}{2}\left(\|\bm{h}\|^2_2 + \|\bm{x}\|^2_2\right).
    \end{align}
    Then, for any $L$ satisfying
    \begin{align}
        \label{ineq:l-smad-blind-deconv-dc}
        L \geq \sum_{j=1}^m(3\|\bm{b}_j\|^4_2+3\|\bm{a}_j\|^4_2&+\|\bm{b}_j\|^2_2\|\bm{a}_j\|^2_2+|y_j|^2\|\bm{b}_j\|^2_2+\|\bm{a}_j\|^2_2),
    \end{align}
    the function $L H - F_1$ is convex on $\real^{d_1+d_2}$.
\end{theorem}
\begin{proof}
We obtain the Hessian of $H$ and $F_1$ as follows:
\begin{align*}
    \nabla^2 H (\bm{z}) &=(\|\bm{z}\|^2_2+1)\bm{I}_{d_1+d_2}+2\bm{z}\bm{z}^{\mathsf{T}},\\
    \nabla^2 F_1 (\bm{h}, \bm{x}) &=
    \re\begin{bmatrix}
        \bm{H}_{11} & \bm{H}_{12}\\
        \bm{H}_{12}^{\mathsf{T}} & \bm{H}_{22}
    \end{bmatrix},
\end{align*}
where $\bm{z}=(\bm{h}, \bm{x})\in\real^{d_1+d_2}$ and
\begin{align*}
    \bm{H}_{11} &:= \bm{B}^{\mathsf{H}}\diag(2|\bm{Bh}|^2 + |\bm{Ax}|^2 + |\bm{y}|^2)\bm{B}+ \bm{B}^{\mathsf{H}}\diag((\bm{Bh})^2)\overline{\bm{B}},\\
   \bm{H}_{12} &:= \bm{B}^{\mathsf{H}}\diag(\bm{Bh}\odot\bm{Ax})\overline{\bm{A}}+\bm{B}^{\mathsf{H}}\diag(\bm{Bh}\odot\overline{\bm{Ax}})\bm{A},\\
    \bm{H}_{22} &:=\bm{A}^{\mathsf{H}}\diag(|\bm{Bh}|^2 + 2|\bm{Ax}|^2 + \bm{1})\bm{A}+ \bm{A}^{\mathsf{H}}\diag((\bm{Ax})^2)\overline{\bm{A}}.
\end{align*}
Since the sum of a complex number and its conjugate is real, $\nabla^2 F_1$ is real.
To prove the convexity of $LH-F_1$,
it is sufficient to show that $\langle\bm{w}, \nabla^2 F_1(\bm{h}, \bm{x})\bm{w}\rangle\leq L\langle\bm{w},\nabla^2 H(\bm{h}, \bm{x})\bm{w}\rangle$ for any $\bm{w}\in\real^{d_1+d_2}$.
Let separate $\bm{w}=(\bm{u}, \bm{v})$ using $\bm{u}\in\real^{d_1}$ and $\bm{v}\in\real^{d_2}$.
We obtain $\langle\bm{w}, \nabla^2 H(\bm{z})\bm{w}\rangle= (\|\bm{z}\|^2_2 + 1)\|\bm{w}\|^2_2 +2\langle\bm{z}, \bm{w}\rangle^2$ and
\begin{align*}
    \langle \bm{w}, \nabla^2 F_1 (\bm{h}, \bm{x})\bm{w} \rangle
    &= \re\langle\bm{u}, \bm{H}_{11}\bm{u}\rangle+ \re\langle\bm{v}, \bm{H}_{22}\bm{v}\rangle +2\re\langle\bm{u}, \bm{H}_{12}\bm{v}\rangle.
\end{align*}
Each term of $\langle \bm{w}, \nabla^2 F_1(\bm{h}, \bm{x})\bm{w} \rangle$ is bounded as follows:
\begin{align*}
    \re\langle\bm{u}, \bm{H}_{11}\bm{u}\rangle
    &=\langle2|\bm{Bh}|^2 + |\bm{Ax}|^2 + |\bm{y}|^2, |\bm{Bu}|^2\rangle+\re\langle(\bm{Bu})^2, (\bm{Bh})^2\rangle\\
    &\leq\langle|\bm{Bh}|^2 + |\bm{Ax}|^2 + |\bm{y}|^2, |\bm{Bu}|^2\rangle+2\langle|\bm{Bh}|^2, |\bm{Bu}|^2\rangle,\\
    \re\langle\bm{v}, \bm{H}_{22}\bm{v}\rangle
    &=\left\langle|\bm{Bh}|^2 + 2|\bm{Ax}|^2 + \bm{1}, |\bm{Av}|^2\right\rangle+\re\langle(\bm{Av})^2, (\bm{Ax})^2\rangle\\
    &\leq\left\langle|\bm{Bh}|^2 + |\bm{Ax}|^2 + \bm{1}, |\bm{Av}|^2\right\rangle+2\left\langle|\bm{Ax}|^2, |\bm{Av}|^2\right\rangle,\\
    \re\langle\bm{u}, \bm{H}_{12}\bm{v}\rangle
    &=\re\langle\bm{u}, \bm{B}^{\mathsf{H}}\diag(\bm{Bh}\odot\bm{Ax})\overline{\bm{A}}\bm{v}\rangle + \re\langle\bm{u}, \bm{B}^{\mathsf{H}}\diag(\bm{Bh}\odot\overline{\bm{A}\bm{x}})\bm{A}\bm{v}\rangle\\
    &=\re\langle\bm{Bu}\odot\bm{Av}, \bm{Bh}\odot\bm{Ax}\rangle + \re\langle\bm{B}\bm{u}\odot\overline{\bm{A}\bm{v}}, \bm{Bh}\odot\overline{\bm{A}\bm{x}}\rangle\\
    &=\re\sum_{j=1}^m\overline{\langle\bm{b}_j, \bm{u}\rangle\langle\bm{a}_j, \bm{v}\rangle}\langle\bm{b}_j, \bm{h}\rangle\langle\bm{a}_j, \bm{x}\rangle+\re\sum_{j=1}^m\overline{\langle\bm{b}_j, \bm{u}\rangle}\langle\bm{a}_j, \bm{v}\rangle\langle\bm{b}_j, \bm{h}\rangle\overline{\langle\bm{a}_j, \bm{x}\rangle}\\
    &\leq2\sum_{j=1}^m\|\bm{b}_j\|^2_2\|\bm{a}_j\|^2_2\|\bm{h}\|_2\|\bm{u}\|_2\|\bm{x}\|_2\|\bm{v}\|_2,
\end{align*}
where all the inequalities hold by $\re(\cdot)\leq|\cdot|$, and the last inequality holds by the Cauchy-Schwarz inequality.
From the above relation, we obtain
\begin{align*}
    \langle|\bm{Bh}|^2&, |\bm{Bu}|^2\rangle + \langle|\bm{Ax}|^2,|\bm{Av}|^2\rangle + \re\langle\bm{u}, \bm{H}_{12}\bm{v}\rangle\\
    &\leq\sum_{j=1}^m\big(\|\bm{b}_j\|^4_2\|\bm{h}\|^2_2\|\bm{u}\|^2_2+\|\bm{a}_j\|^4_2\|\bm{x}\|^2_2\|\bm{v}\|^2_2 + 2\|\bm{b}_j\|^2_2\|\bm{a}_j\|^2_2\|\bm{h}\|_2\|\bm{u}\|_2\|\bm{x}\|_2\|\bm{v}\|_2\big)\\
    &=\sum_{j=1}^m\big(\|\bm{b}_j\|^2_2\|\bm{h}\|_2\|\bm{u}\|_2+\|\bm{a}_j\|^2_2\|\bm{x}\|_2\|\bm{v}\|_2\big)^2\\
    &\leq\sum_{j=1}^m\big(\|\bm{b}_j\|^4_2+\|\bm{a}_j\|^4_2\big)\big(\|\bm{h}\|^2_2\|\bm{u}\|^2_2+\|\bm{x}\|^2_2\|\bm{v}\|^2_2\big),
\end{align*}
where both inequalities hold by the Cauchy-Schwarz inequality.
Thus, we obtain
\begin{align*}
    &\langle \bm{w}, \nabla^2 F_1(\bm{h}, \bm{x})\bm{w} \rangle\\
    &\leq\langle|\bm{Bh}|^2 + |\bm{Ax}|^2 + |\bm{y}|^2, |\bm{Bu}|^2\rangle+\langle|\bm{Bh}|^2 + |\bm{Ax}|^2 + \bm{1}, |\bm{Av}|^2\rangle\\
    &\quad+2\langle|\bm{Bh}|^2, |\bm{Bu}|^2\rangle+2\langle|\bm{Ax}|^2, |\bm{Av}|^2\rangle+2\re\langle\bm{u}, \bm{H}_{12}\bm{v}\rangle\\
    &\leq\sum_{j=1}^m\big(\|\bm{b}_j\|^2_2\|\bm{u}\|^2_2(\|\bm{b}_j\|^2_2\|\bm{h}\|^2_2+ \|\bm{a}_j\|^2_2\|\bm{x}\|^2_2+|y_j|^2)\\
    &\quad+\|\bm{a}_j\|^2_2\|\bm{v}\|^2_2(\|\bm{b}_j\|^2_2\|\bm{h}\|^2_2 + \|\bm{a}_j\|^2_2\|\bm{x}\|^2_2+1)\\
    &\quad+2(\|\bm{b}_j\|^4_2+\|\bm{a}_j\|^4_2)(\|\bm{h}\|^2_2\|\bm{u}\|^2_2+\|\bm{x}\|^2_2\|\bm{v}\|^2_2)\big)\\
    &\leq\sum_{j=1}^m(3\|\bm{b}_j\|^4_2+3\|\bm{a}_j\|^4_2+\|\bm{b}_j\|^2_2\|\bm{a}_j\|^2_2
    +|y_j|^2\|\bm{b}_j\|^2_2+\|\bm{a}_j\|^2_2)(\|\bm{z}\|^2_2 + 1)\|\bm{w}\|^2_2\\
    &\leq L\langle\bm{w}, \nabla^2 H(\bm{h}, \bm{x})\bm{w}\rangle,
\end{align*}
which proves $L H - F_1$ is convex.
\end{proof}
From Theorem~\ref{theo:l-smad-para}, we obtain the following.
\begin{corollary}
\label{coro:l-smad}
Let $H_C\in\mathcal{G}(C)$ be defined by 
$H_C(\bm{z}) = \frac{1}{4}\|\bm{z}\|^4_2 + \frac{1}{2}\|\bm{z}\|^2_2 +\delta_C(\bm{z})$,
where $\delta_C(\bm{z})$ is the indicator function $\delta_C(\bm{z}) = 0$ for $\bm{z} \in C$ and $\delta_C(\bm{z}) = \infty$ otherwise.
For any $L$ satisfying~\eqref{ineq:l-smad-blind-deconv-dc}, the pair $(F_1, H_C)$ is $L$-smad on $C$.
\end{corollary}
\begin{proof}
Because $C$ is an open set, $H_C$ and $\nabla^2 H_C$ is the same to these of $H$ given by~\eqref{def:kernel-blind-dc} on $C$.
From the convexity of $C$ and Theorem~\ref{theo:l-smad-para}, the pair $(F_1, H_C)$ is $L$-smad on $C$.
\end{proof}
From Corollary~\ref{coro:l-smad}, we can immediately prove the following corollary by using~\cite[Theorems 2 and 7]{takahashi21}.
\begin{corollary}
\label{coro:convergence}
Let $\{\bm{z}^k\}$ be a sequence generated by BPDCA(e) with $0 < \lambda L < 1$ for~\eqref{prob:blind-deconv-dc}.
For BPDCAe, assume $G$ is convex.
Then, $\{\bm{z}^k\}$ converges to a stationary point of~\eqref{prob:blind-deconv-dc}.
\end{corollary}
Note that a stationary point of~\eqref{prob:blind-deconv-dc} is defined as a point $\bm{z} \in C$ satisfying $-\nabla F_1(\bm{z}) + \nabla F_2(\bm{z})\in\partial G(\bm{z})$.
Note that under the convexity of $G$, it is theoretically guaranteed that a sequence generated by FISTA would converge to the stationary point if $F$ had a Lipschitz continuous gradient.
In our problem, the convergence of FISTA is not theoretically guaranteed because $F$ does not have it.
Although FISTA is not applicable, the convergent points of BPDCA(e) and FISTA share the same stationary points in theory.

\begin{remark}
An appropriate value of $\lambda$ can be evaluated by backtracking, i.e., decrease $\lambda$ until $F_1(\bm{z}^{k+1}) - F_1(\bm{z}^k) - \langle\nabla F_1(\bm{z}^k), \bm{z}^{k+1} - \bm{z}^k\rangle \leq \frac{1}{\lambda}D_H(\bm{z}^{k+1}, \bm{z}^k)$ is satisfied, because, if this inequality holds, the theoretical convergence is guaranteed.
For $n_k$ backtracking procedures at the $k$th iteration, we need one calculation of $\nabla F_1(\bm{z}^k)$ and $n_k$ calculations of $F_1(\bm{z}^{k+1})$.
There calculations are sometimes expensive.
Thus, we did not use the backtracking.

\end{remark}

\subsection{Stability Analysis}
\label{sec:stability}
We show the stability analysis of BPDCA(e), FISTA, and the alternating minimization (AM)~\cite{chan00,krishnan11,perrone14} around the equilibrium points, which are the fixed points of the update formula of each algorithm.
Here, we define $\Phi(\bm{z}) = G(\bm{z}) + \langle\nabla F(\bm{z}^k), \bm{z}\rangle + \frac{1}{\lambda}D_H(\bm{z}, \bm{z}^k)$.
Then, the first-order condition $\bm{0}\in\partial\Phi(\bm{z}^{k+1}) = \partial G(\bm{z}^{k+1}) + \nabla F(\bm{z}^k) + \frac{1}{\lambda}\left(\nabla H(\bm{z}^{k+1}) - \nabla H(\bm{z}^k)\right)$ is approximated as follows:
\begin{align*}
    \bm{0}\simeq \bm{\zeta}^{k+1} + \nabla F(\bm{z}^k) + \frac{1}{\lambda}\nabla^2 H(\bm{z}^k)(\bm{z}^{k+1} - \bm{z}^k),
\end{align*}
where $\bm{\zeta}^{k+1}\in\partial G(\bm{z}^{k+1})$.
Note that $\nabla^2 H(\bm{z}^k)$ is regular.
In fact, its inverse is explicitly written as
\begin{align*}
\nabla^2 H(\bm{z})^{-1} = \frac{1}{\|\bm{z}\|^2_2+1}\left(\bm{I}_{d_1+d_2} - \frac{2\bm{z}\bm{z}^{\mathsf{T}}}{3\|\bm{z}\|^2_2+1}\right).
\end{align*}
By multiplying it, we obtain
\begin{align*}
    \bm{z}^{k+1} - \bm{z}^k \simeq -\lambda\nabla^2 H(\bm{z}^k)^{-1}(\bm{\zeta}^{k+1} + \nabla F(\bm{z}^k)),
\end{align*}
which indicates that $\bm{z}^{k+1} - \bm{z}^k$ is greatly affected by $\nabla^2 H(\bm{z})^{-1}$.
Thus, $H$ is important for the performance of BPDCA.
For BPDCAe, this fact is also true by substituting $\bm{w}^k$ for $\bm{z}^k$.

To simplify the stability analysis of FISTA, we consider ISTA, which is FISTA without extrapolation.
Setting $H(\bm{z}) = \frac{1}{2}\|\bm{z}\|^2_2 =: H_1(\bm{z})$, \ie, $\nabla^2 H_1(\bm{z}) = \bm{I}_{d_1+d_2}$, we obtain
\begin{align*}
    \bm{z}^{k+1} - \bm{z}^k \simeq -\lambda (\bm{\zeta}^{k+1} + \nabla F(\bm{z}^k)).
\end{align*}
Since $F$ does not have a Lipschitz continuous gradient, $\lambda$ is close to $0$, \ie, $\bm{z}^k \simeq \bm{z}^{k+1}$.
This implies that convergence of (F)ISTA is slow.

When $G(\bm{x},\bm{h})$ is convex,~\eqref{prob:blind-deconv} is convex with respect to $\bm{h}$ for fixed $\bm{x}$ and vice versa. 
AM is a method to update $\bm{x}$ and $\bm{h}$ alternately, which will be easily implemented and converge. 
The first-order conditions around the equilibrium points are
\begin{align*}
    \bm{0} &\in \partial_{\bm{h}} G(\bm{h}^{k+1}, \bm{x}^{k}) + \nabla_{\bm{h}} F(\bm{h}^{k+1}, \bm{x}^k)\\
    &\simeq \partial_{\bm{h}} G(\bm{h}^{k+1}, \bm{x}^{k}) + \nabla_{\bm{h}} F(\bm{h}^{k}, \bm{x}^k) + \nabla^2_{\bm{hh}} F(\bm{h}^k, \bm{x}^k)(\bm{h}^{k+1} - \bm{h}^k),\\
    \bm{0} &\in \partial_{\bm{x}} G(\bm{h}^{k+1}, \bm{x}^{k+1}) + \nabla_{\bm{x}} F(\bm{h}^{k+1}, \bm{x}^{k+1})\\
    &\simeq \partial_{\bm{x}} G(\bm{h}^{k+1}, \bm{x}^{k+1}) + \nabla_{\bm{x}} F(\bm{h}^{k+1}, \bm{x}^k) + \nabla^2_{\bm{xx}} F(\bm{h}^{k}, \bm{x}^k)(\bm{x}^{k+1} - \bm{x}^k),
\end{align*}
where the last approximation holds from $\bm{h}^{k+1}\simeq\bm{h}^k$ in $\nabla^2_{\bm{xx}} F(\bm{h}^{k}, \bm{x}^k)$.
Assuming $\nabla^2_{\bm{hh}} F(\bm{h}^k, \bm{x}^k)$ and $\nabla^2_{\bm{xx}} F(\bm{h}^{k}, \bm{x}^k)$ are regular, for $\bm{\zeta}^{k+1}_{\bm{h}}\in\partial_{\bm{h}} G(\bm{h}^{k+1}, \bm{x}^k)$ and $\bm{\zeta}^{k+1}_{\bm{x}}\in\partial_{\bm{x}} G(\bm{h}^{k+1}, \bm{x}^{k+1})$, we obtain
\begin{align*}
    \bm{h}^{k+1} - \bm{h}^k &\simeq - \nabla^2_{\bm{hh}} F(\bm{h}^k, \bm{x}^k)^{-1}(\bm{\zeta}^{k+1}_{\bm{h}} + \nabla_{\bm{h}} F(\bm{h}^k, \bm{x}^k)),\\
    \bm{x}^{k+1} - \bm{x}^k &\simeq - \nabla^2_{\bm{xx}} F(\bm{h}^{k}, \bm{x}^k)^{-1}(\bm{\zeta}^{k+1}_{\bm{x}} + \nabla_{\bm{x}} F(\bm{h}^{k+1}, \bm{x}^k)).
\end{align*}
$\nabla^2 H(\bm{z}^k)$ in the approximation of BPDCA is a block matrix that contains the cross derivative terms $\nabla^2_{\bm{hx}}H(\bm{z}^k)$ and $\nabla^2_{\bm{xh}} H(\bm{z}^k)$, while the perturbation around the equilibrium points of AM is approximated only with $\nabla^2_{\bm{xx}} F(\bm{h}^{k}, \bm{x}^k)$ and $\nabla^2_{\bm{hh}} F(\bm{h}^k, \bm{x}^k)$ regardless of the cross derivatives.
Thus, an equilibrium point of AM is not necessarily that of BPDCA.
This implies that BPDCA is not trapped at the points where AM is stuck.
On the other hand, every equilibrium point of BPDCA is an equilibrium point of AM under the Clarke regularity as we mention below.
Here, we assume that $G$ is Clarke regular at $\bm{z} = (\bm{h}, \bm{x})\in\real^{d_1+d_2}$.
Then, it holds that $\partial G(\bm{h}, \bm{x}) \subset\partial_{\bm{h}} G(\bm{h}, \bm{x})\times\partial_{\bm{x}} G(\bm{h}, \bm{x})$~\cite[Proposition 2.3.15]{clarke90}.
Hence, for an equilibrium point $\bm{z} = (\bm{h}, \bm{x})$ of BPDCA, we have
\begin{align*}
\bm{0} &\in \nabla F (\bm{z}) + \partial G(\bm{z}) + \nabla H (\bm{z}) - \nabla H (\bm{z})\\
&\subset(\nabla_{\bm{h}} F (\bm{h}, \bm{x}), \nabla_{\bm{x}} F (\bm{h}, \bm{x})) + \partial_{\bm{h}} G(\bm{h}, \bm{x})\times\partial_{\bm{x}} G(\bm{h}, \bm{x}).
\end{align*}
Therefore, an equilibrium point
of BPDCA is also an equilibrium point of AM.
They are demonstrated in numerical experiments (Figure~\ref{fig:compare-recover}).

\section{Numerical Experiments}
We demonstrated efficiency of our proposed method by image deblurring via solving problem~\eqref{prob:blind-deconv}.
We set $d_1 = 2304$, $d_2 = 65536$, and $m = 262144$ and appropriately took a ground truth of $(\bm{h}^{\circ}, \bm{x}^{\circ})$. 
Using them, we generated a blurring kernel $\bm{f} = \tilde{\bm{B}}\bm{h}^{\circ}$ and an original image $\bm{g} = \tilde{\bm{A}}\bm{x}^{\circ}$, where $\tilde{\bm{B}}: \real^{d_1} \to \real^{m}$ is an operator reshaping $\bm{h} \in \real^{d_1}$ into a $\sqrt{m} \times \sqrt{m}$ image and $\tilde{\bm{A}}: \real^{d_2} \to \real^{m}$ is an inverse discrete Meyer wavelet transform operator.
Figure~\ref{fig:recover} depicts $\bm{f}$ and $\bm{g}$ used in our experiments: $\bm{f}$ corresponds to a diagonal blurring and $\bm{g}$ approximates a natural image. 
We also set $C=\{(\bm{h},\bm{x})\in\real^{d_1}\times\real^{d_2}\ |\ \bm{h}>\bm{0}, \bm{x}>\bm{0}\}$ and $G(\bm{h}, \bm{x}) = \theta \|\bm{h}\|_1$ with $\theta=0.01$ (the non-smooth $\ell_1$ regularizer) because $\bm{h}$ is supposed to be sparse in practice of image deblurring.

\begin{figure}[b]
    \centering
    \includegraphics[width=\textwidth]{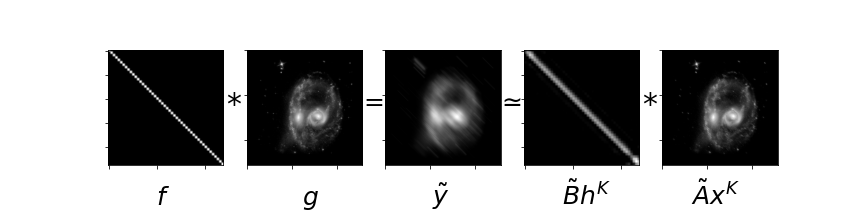}
    \caption{The ground truth $\bm{f}$ and $\bm{g}$, the blurred image $\tilde{\bm{y}}$, and $\bm{Bh}^K$ and $\bm{Ax}^K$ recovered by BPDCAe.}
    \label{fig:recover}
\end{figure}

\subsection{Comparison of $\ell_1$ and $\ell_2$ Regularization}
We solved problem~\eqref{prob:blind-deconv} corresponding to the setting above with BPDCA(e), FISTA, and AM.
For BPDCA(e), we adjusted $L$ which satisfies~\eqref{ineq:l-smad-blind-deconv-dc} and used it as a fixed step size.
Step sizes in all iterations of FISTA were obtained by backtracking.
Note that the subproblems of BPDCA(e) and FISTA (without backtracking procedures) are solved in closed-form formulae, whose computational cost is almost the same.
The maximum number of iterations for BPDCA(e) and FISTA was 30000, and that for AM was 3000 because the subproblems of AM were solved approximately by 10 iterations of FISTA at each iteration.
For all methods, the initial points $\bm{h}^0$ and $\bm{x}^0$ are set to be the left and right singular vectors corresponding to the leading singular value of $\bm{B}^{\mathsf{H}}\diag(\bm{y})\overline{\bm{A}}$, respectively, which is proposed in~\cite{li19}.
The difference between the objective value at each iteration and that at the ground truth $\{\log_{10}|\Psi(\bm{h}^k,\bm{x}^k) - \Psi^{\circ}|\}$ is plotted in Figure~\ref{fig:objective-log-diff} in log-scale, where we denoted $\Psi = F + G$ and $\Psi^{\circ} := \Psi(\bm{h}^{\circ}, \bm{x}^{\circ})$.
Figure~\ref{fig:cos-sim-h} shows the cosine similarity between $\bm{h}^k$ and $\bm{h}^{\circ}$ defined by $\{\cossim (\bm{h}^k, \bm{h}^{\circ})(:=\langle\bm{h}^k, \bm{h}^{\circ}\rangle/(\|\bm{h}^k\|_2\|\bm{h}^{\circ}\|_2))\}$, and Figure~\ref{fig:cos-sim-x} shows that $\cossim (\bm{x}^k, \bm{x}^{\circ})$.
As we can see from Figures~\ref{fig:objective-log-diff} and~\ref{fig:cos-sim}, BPDCAe outperformed the other algorithms.
Its convergence was the fastest, and $\Psi(\bm{h}^K, \bm{x}^K)$, $\cossim (\bm{h}^K, \bm{h}^{\circ})$, and $\cossim (\bm{x}^K, \bm{x}^{\circ})$ were also the best, where $K := 30000$ (for BPDCA(e) and FISTA) or $K := 3000$ (for AM).
Figure~\ref{fig:compare-recover} shows the recovered images.
Figure~\ref{fig:bpdcae} shows that there is almost no difference between $\tilde{\bm{A}}\bm{x}^{\circ}$ and $\tilde{\bm{A}}\bm{x}^k$, while $\bm{h}^k$ was not completely recovered.
Figures~\ref{fig:bpdcae} and~\ref{fig:am} show that the sequences generated by AM converged to a different stationary point (see also Subsection~\ref{sec:stability}).
\begin{figure}[t]
    \centering
    \includegraphics[width=.9\textwidth]{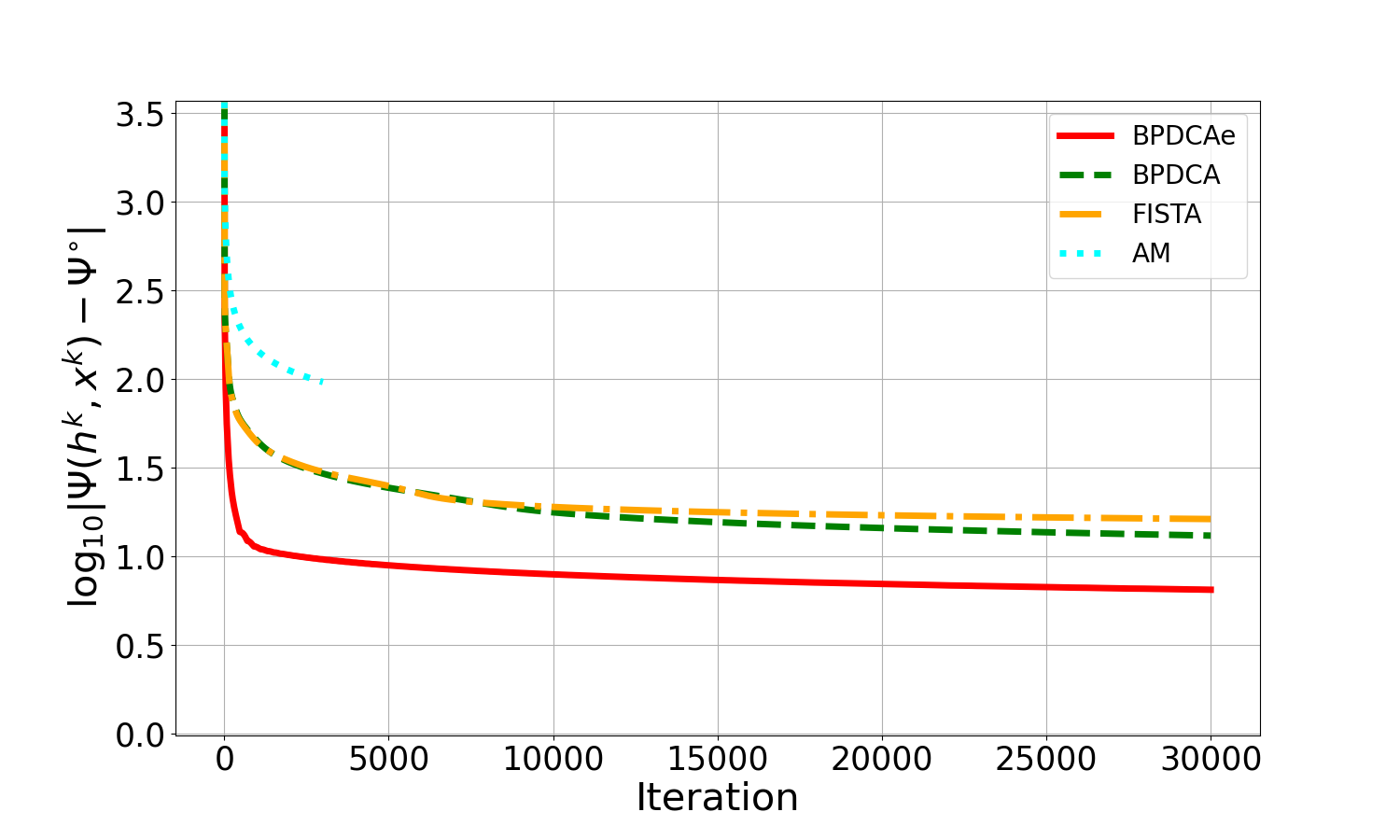}
    \caption{Plots of $\{\log_{10}|\Psi(\bm{h}^k, \bm{x}^k)-\Psi^{\circ}|\}$ at each iteration.}
    \label{fig:objective-log-diff}
\end{figure}
\begin{figure}[t]
\begin{minipage}[b]{0.49\linewidth}
    \centering
    \includegraphics[width=\textwidth]{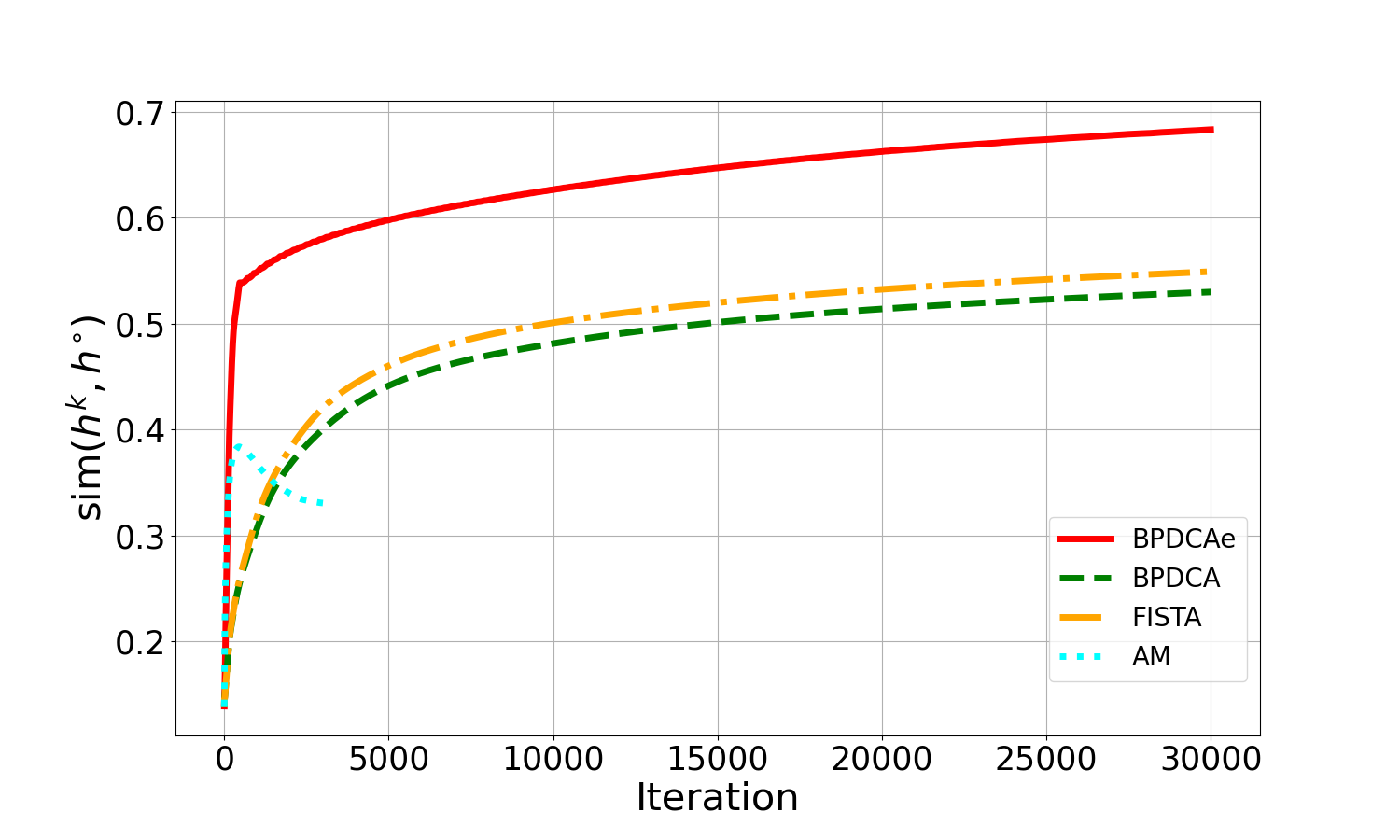}
    \subcaption{Plots of $\{\cossim (\bm{h}^k, \bm{h}^{\circ})\}$.}
    \label{fig:cos-sim-h}
\end{minipage}
\hfill
\begin{minipage}[b]{0.49\linewidth}
    \centering
    \includegraphics[width=\textwidth]{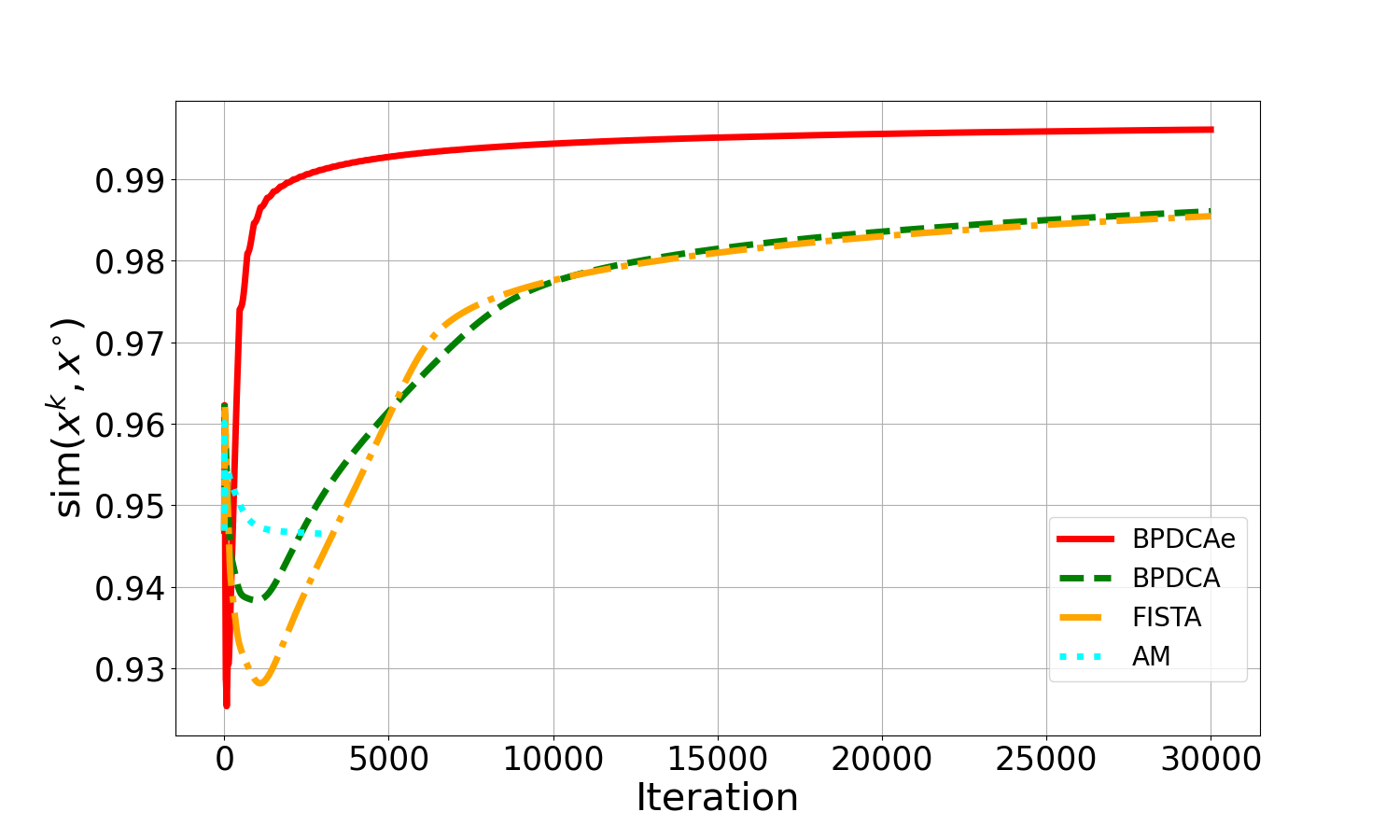}
    \subcaption{Plots of $\{\cossim (\bm{x}^k, \bm{x}^{\circ})\}$.}
    \label{fig:cos-sim-x}
\end{minipage}
\caption{Plots of the cosine similarities between the $k$th point and the ground truth.}
\label{fig:cos-sim}
\end{figure}

\begin{figure}[t]
\begin{minipage}[b]{0.16\linewidth}
    \centering
    \includegraphics[width=\textwidth]{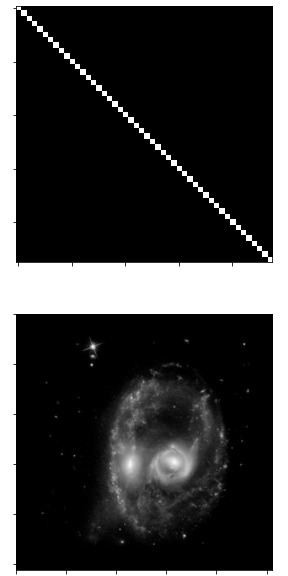}
    \subcaption{$\bm{h}^{\circ}$ and $\bm{x}^{\circ}$}
    \label{fig:ground-truth}
\end{minipage}
\hfill
\begin{minipage}[b]{0.16\linewidth}
    \centering
    \includegraphics[width=\textwidth]{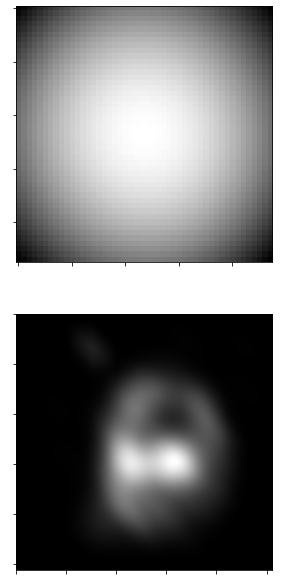}
    \subcaption{$\bm{h}^0$ and $\bm{x}^0$}
    \label{fig:init}
\end{minipage}
\hfill
\begin{minipage}[b]{0.16\linewidth}
    \centering
    \includegraphics[width=\textwidth]{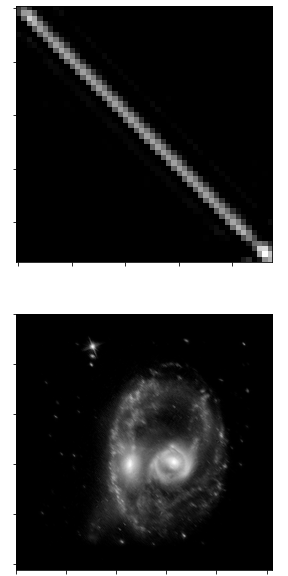}
    \subcaption{BPDCAe}
    \label{fig:bpdcae}
\end{minipage}
\hfill
\begin{minipage}[b]{0.16\linewidth}
    \centering
    \includegraphics[width=\textwidth]{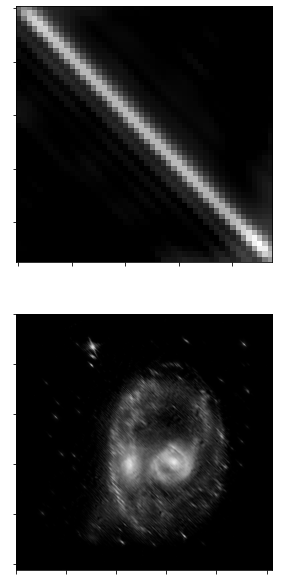}
    \subcaption{BPDCA}
    \label{fig:bpdca}
\end{minipage}
\hfill
\begin{minipage}[b]{0.16\linewidth}
    \centering
    \includegraphics[width=\textwidth]{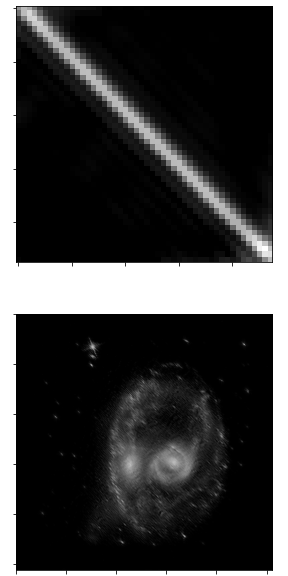}
    \subcaption{FISTA}
    \label{fig:fista}
\end{minipage}
\hfill
\begin{minipage}[b]{0.16\linewidth}
    \centering
    \includegraphics[width=\textwidth]{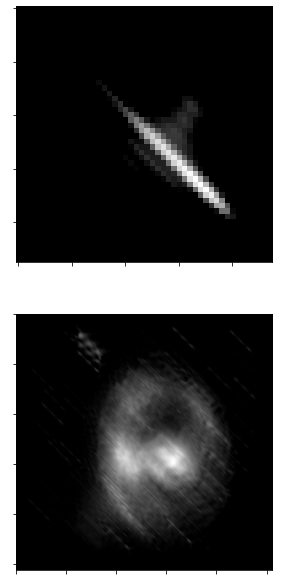}
    \subcaption{AM}
    \label{fig:am}
\end{minipage}
\caption{The upper row shows $\bm{h}^K$, and the lower row shows $\tilde{\bm{A}}\bm{x}^K$.}
\label{fig:compare-recover}
\end{figure}

We also solved the deblurring problem with the $\ell_2$ regularizer $G(\bm{h}, \bm{x}) = \theta\|\bm{h}\|_2^2$ with $\theta=0.01$.
The comparison between the results from these two regularizers are shown in Figure~\ref{fig:vs-l2}.
It shows the superiority of the non-smooth $\ell_1$ regularization over the $\ell_2$ one, which did not recover the sparse blurring kernel.

\begin{figure}[t]
\begin{minipage}[b]{0.19\linewidth}
    \centering
    \includegraphics[width=\textwidth]{recover-result-bpdcae.png}
    \subcaption{$\theta\|\bm{h}\|_1$}
    \label{fig:bpdcael1}
\end{minipage}
\hfill
\begin{minipage}[b]{0.19\linewidth}
    \centering
    \includegraphics[width=\textwidth]{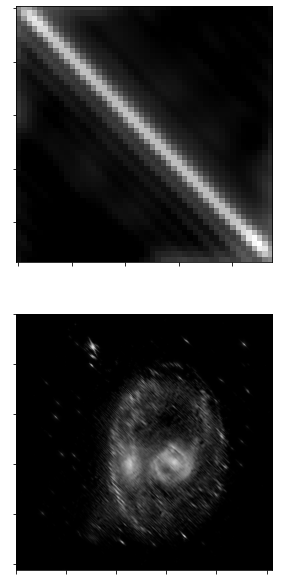}
    \subcaption{$\theta\|\bm{h}\|_2^2$}
    \label{fig:bpdcael2}
\end{minipage}
\hfill
\begin{minipage}[b]{0.19\linewidth}
    \centering
    \includegraphics[width=\textwidth]{recover-result-bpdca.png}
    \subcaption{$\theta\|\bm{h}\|_1$}
    \label{fig:bpdcal1}
\end{minipage}
\hfill
\begin{minipage}[b]{0.19\linewidth}
    \centering
    \includegraphics[width=\textwidth]{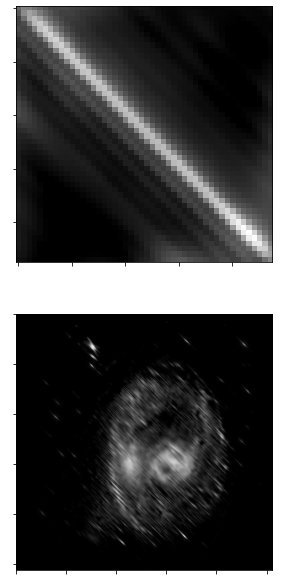}
    \subcaption{$\theta\|\bm{h}\|_2^2$}
    \label{fig:bpdcal2}
\end{minipage}
\caption{(a--b) BPDCAe and (c--d) BPDCA.}
\label{fig:vs-l2}
\end{figure}

\subsection{Comparisons under Several Situations}
\begin{figure}
    \centering
    \includegraphics[width=.9\textwidth]{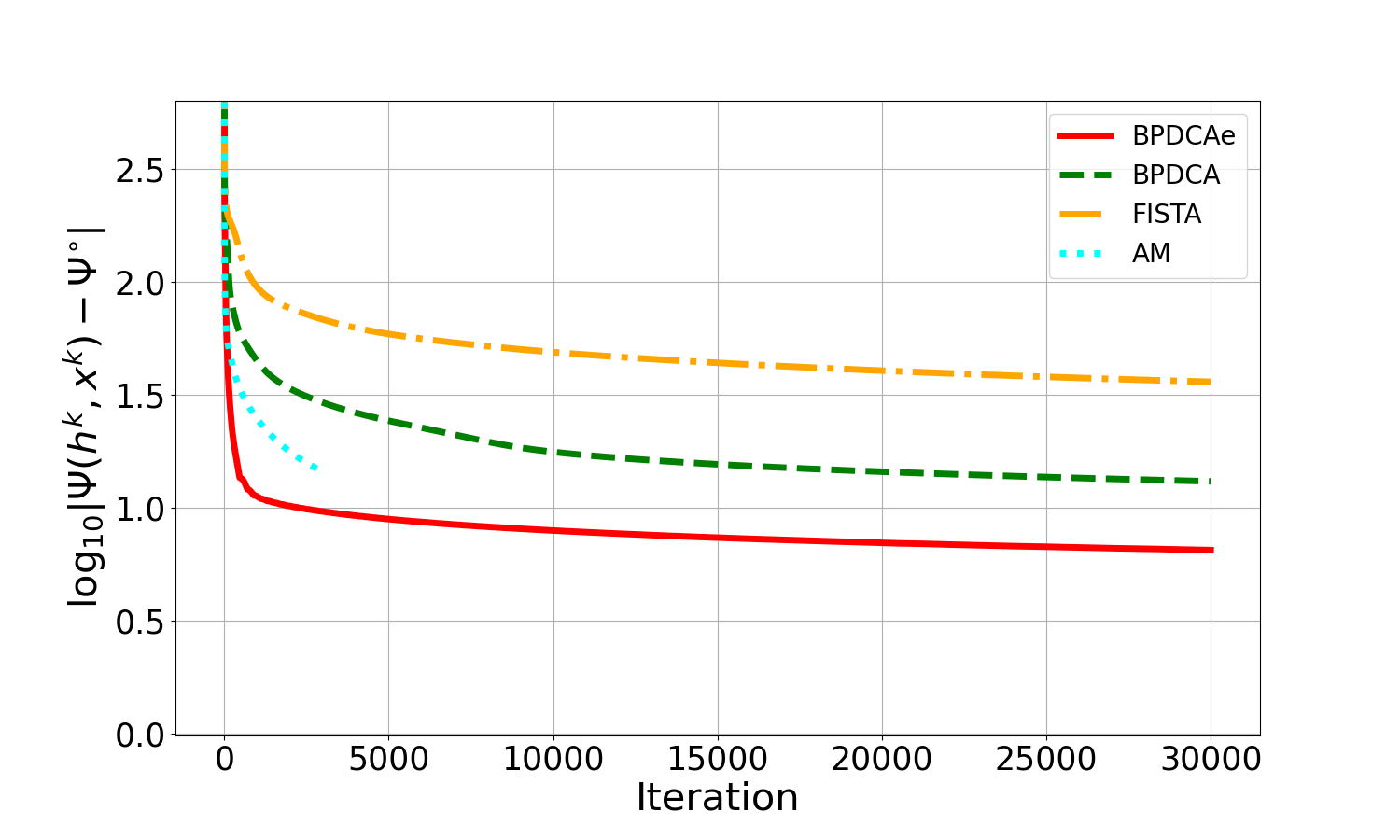}
    \caption{Plots of $\{\log_{10}|\Psi(\bm{h}^k, \bm{x}^k)-\Psi^{\circ}|\}$ when $\bm{z}^0$ is generated from the uniform distribution.}
    \label{fig:objective-log-diff-random-init}
\end{figure}

We first demonstrate the efficiency of our proposed methods is independent of the choice of the initial point.
To do so, we generated the initial points $\bm{h}^0$ and $\bm{x}^0$ from the uniform distribution on $[0, 0.1]$.
From these initial points, we generated $\{\bm{h}^k\}$ and $\{\bm{x}^k\}$ by each algorithm.
Figures~\ref{fig:objective-log-diff-random-init},~\ref{fig:cos-sim-h-random-init}, and~\ref{fig:cos-sim-x-random-init} show $\{\log_{10}|\Psi(\bm{h}^k,\bm{x}^k) - \Psi^{\circ}|\}$, $\{\cossim (\bm{h}^k, \bm{h}^{\circ})\}$, and $\{\cossim (\bm{x}^k, \bm{x}^{\circ})\}$, respectively.
Figure~\ref{fig:compare-recover-random} shows the recovered images, and Figure~\ref{fig:bpdcae-random} shows that there is almost no difference between $\tilde{\bm{A}}\bm{x}^{\circ}$ and $\tilde{\bm{A}}\bm{x}^K$ at this case.
As we can see from these figures, BPDCAe also outperformed the other algorithms even when $\bm{h}^0$ and $\bm{x}^0$ are generated from the uniform distribution.

\begin{figure}
\begin{minipage}[b]{0.49\linewidth}
    \centering
    \includegraphics[width=\textwidth]{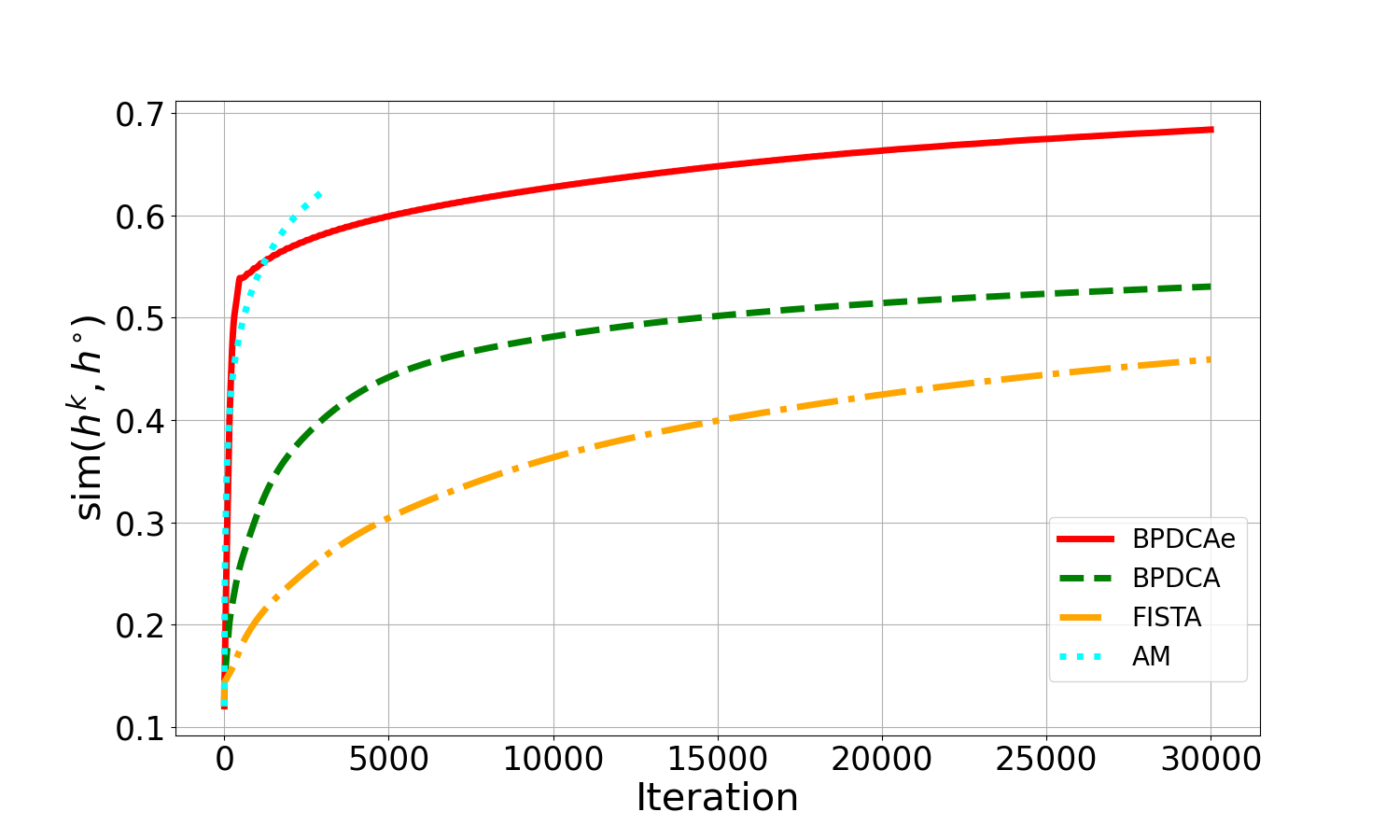}
    \subcaption{Plots of $\{\cossim (\bm{h}^k, \bm{h}^{\circ})\}$.}
    \label{fig:cos-sim-h-random-init}
\end{minipage}
\hfill
\begin{minipage}[b]{0.49\linewidth}
    \centering
    \includegraphics[width=\textwidth]{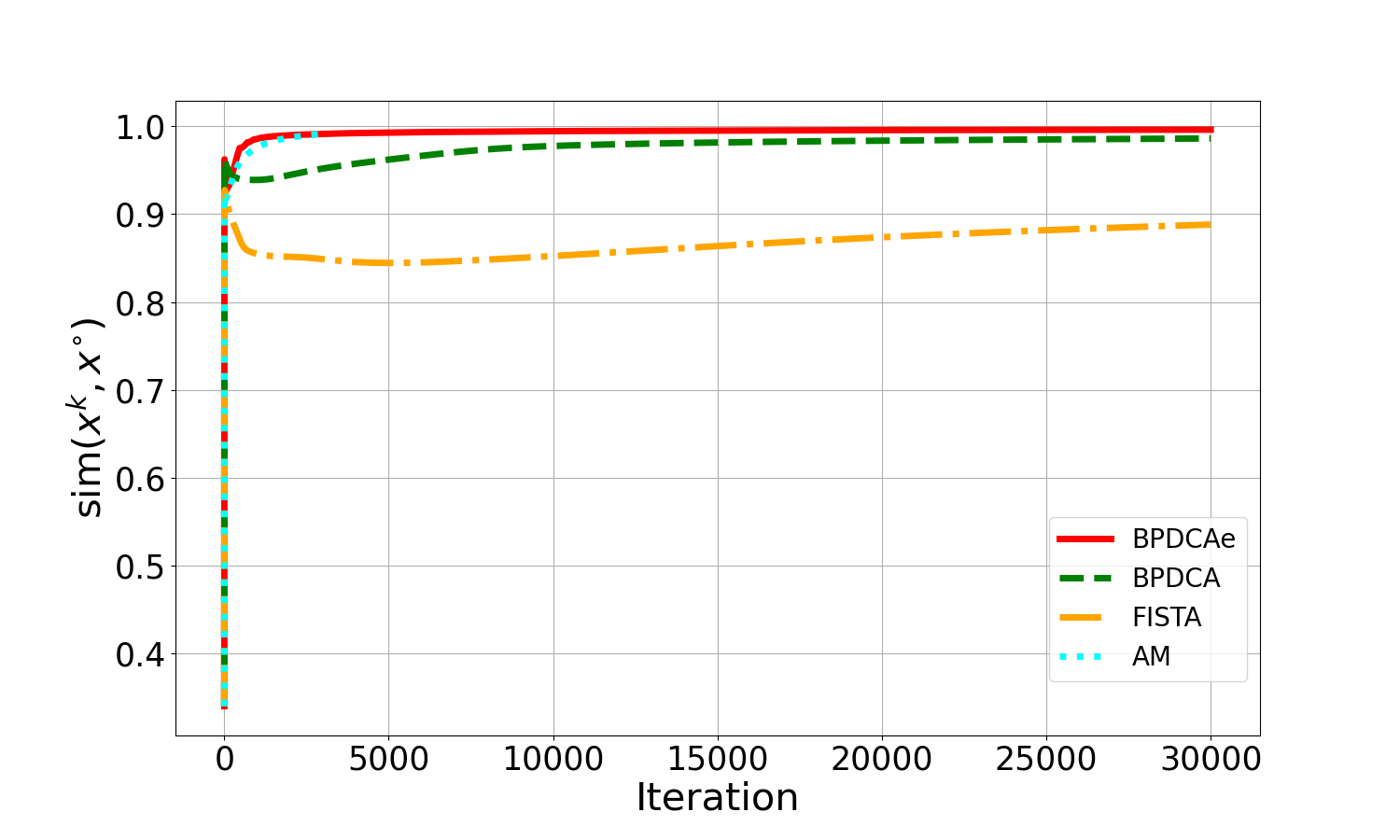}
    \subcaption{Plots of $\{\cossim (\bm{x}^k, \bm{x}^{\circ})\}$.}
    \label{fig:cos-sim-x-random-init}
\end{minipage}
\caption{Plots of the cosine similarities when $\bm{z}^0$ is generated from the uniform distribution.}
\label{fig:cos-sim-random}
\end{figure}

\begin{figure}
\begin{minipage}[b]{0.16\linewidth}
    \centering
    \includegraphics[width=\textwidth]{recover-result-gt.png}
    \subcaption{$\bm{h}^{\circ}$ and $\bm{x}^{\circ}$}
    \label{fig:gt-random-init}
\end{minipage}
\hfill
\begin{minipage}[b]{0.16\linewidth}
    \centering
    \includegraphics[width=\textwidth]{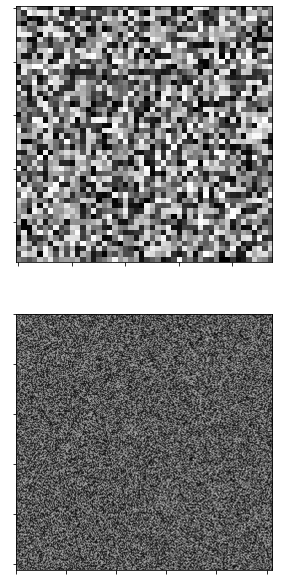}
    \subcaption{$\bm{h}^0$ and $\bm{x}^0$}
    \label{fig:random-init}
\end{minipage}
\hfill
\begin{minipage}[b]{0.16\linewidth}
    \centering
    \includegraphics[width=\textwidth]{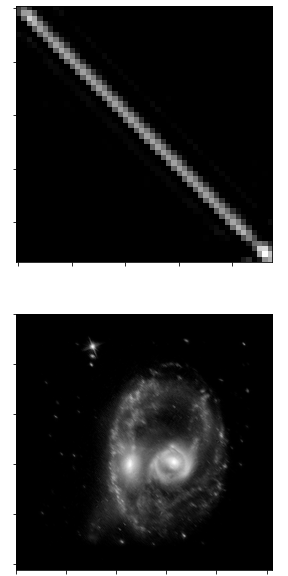}
    \subcaption{BPDCAe}
    \label{fig:bpdcae-random}
\end{minipage}
\hfill
\begin{minipage}[b]{0.16\linewidth}
    \centering
    \includegraphics[width=\textwidth]{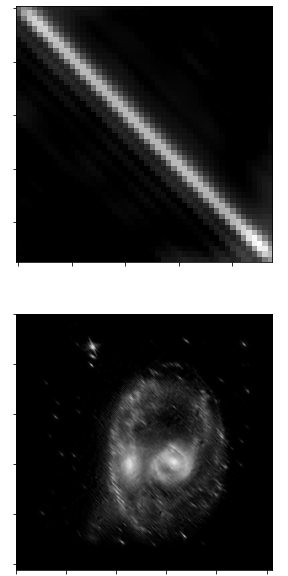}
    \subcaption{BPDCA}
    \label{fig:bpdca-random}
\end{minipage}
\hfill
\begin{minipage}[b]{0.16\linewidth}
    \centering
    \includegraphics[width=\textwidth]{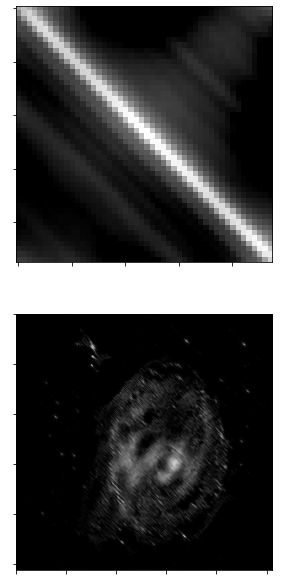}
    \subcaption{FISTA}
    \label{fig:fista-random}
\end{minipage}
\hfill
\begin{minipage}[b]{0.16\linewidth}
    \centering
    \includegraphics[width=\textwidth]{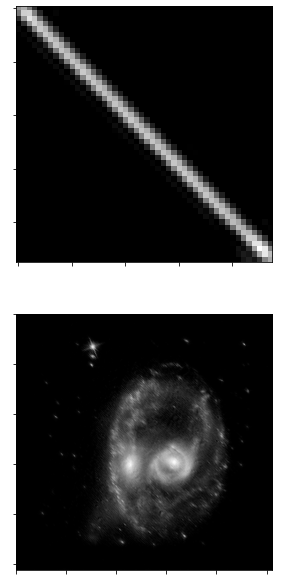}
    \subcaption{AM}
    \label{fig:am-random}
\end{minipage}
\caption{The upper row shows $\bm{h}^K$, and the lower row shows $\tilde{\bm{A}}\bm{x}^K$ when $\bm{z}^0$ is generated from the uniform distribution.}
\label{fig:compare-recover-random}
\end{figure}

We next demonstrate the efficiency of our proposed methods with the noisy data.
Here, we consider $\tilde{\bm{y}}$ containing Poisson noise, \ie, $\tilde{\bm{y}} = \bm{f} * \bm{g} + \bm{n}$, where $\bm{n}\in\real^m$ is Poisson noise (see Figure~\ref{fig:noisy-image}).
By changing noise level, \ie, the standard deviation, we solved the image deblurring problem with each algorithm.
Figure~\ref{fig:result-noise} shows the objective value $\Psi(\bm{h}^K,\bm{x}^K)$, the loss function value $F(\bm{h}^K,\bm{x}^K)$, and the cosine similarities $\cossim (\bm{h}^K, \bm{h}^{\circ})$ and $\cossim (\bm{x}^K, \bm{x}^{\circ})$, where $\bm{h}^K$ and $\bm{x}^K$ are recovered by each algorithm after $K=30000$ (for BPDCA(e) and FISTA) or $K=3000$ (for AM) iterations for each noise level.
Whereas the cosine similarity $\cossim (\bm{h}^K, \bm{h}^{\circ})$ recovered by FISTA was the best among Figure~\ref{fig:cos-sim-h-noise}, the cosine similarity $\cossim (\bm{x}^K, \bm{x}^{\circ})$ recovered by BPDCAe was the best among Figure~\ref{fig:cos-sim-x-noise}.
Figure~\ref{fig:compare-recover-noise} shows the recovered images when the standard deviation of $\bm{n}$ is 10.6.
$\bm{x}^K$ recovered by BPDCAe is the best, and the objective value by BPDCAe is the smallest in Figure~\ref{fig:result-noise}.
Thus, BPDCA(e) outperformed the other algorithms with the noisy data.

\begin{figure}
\begin{minipage}[b]{0.49\linewidth}
    \centering
    \includegraphics[width=\textwidth]{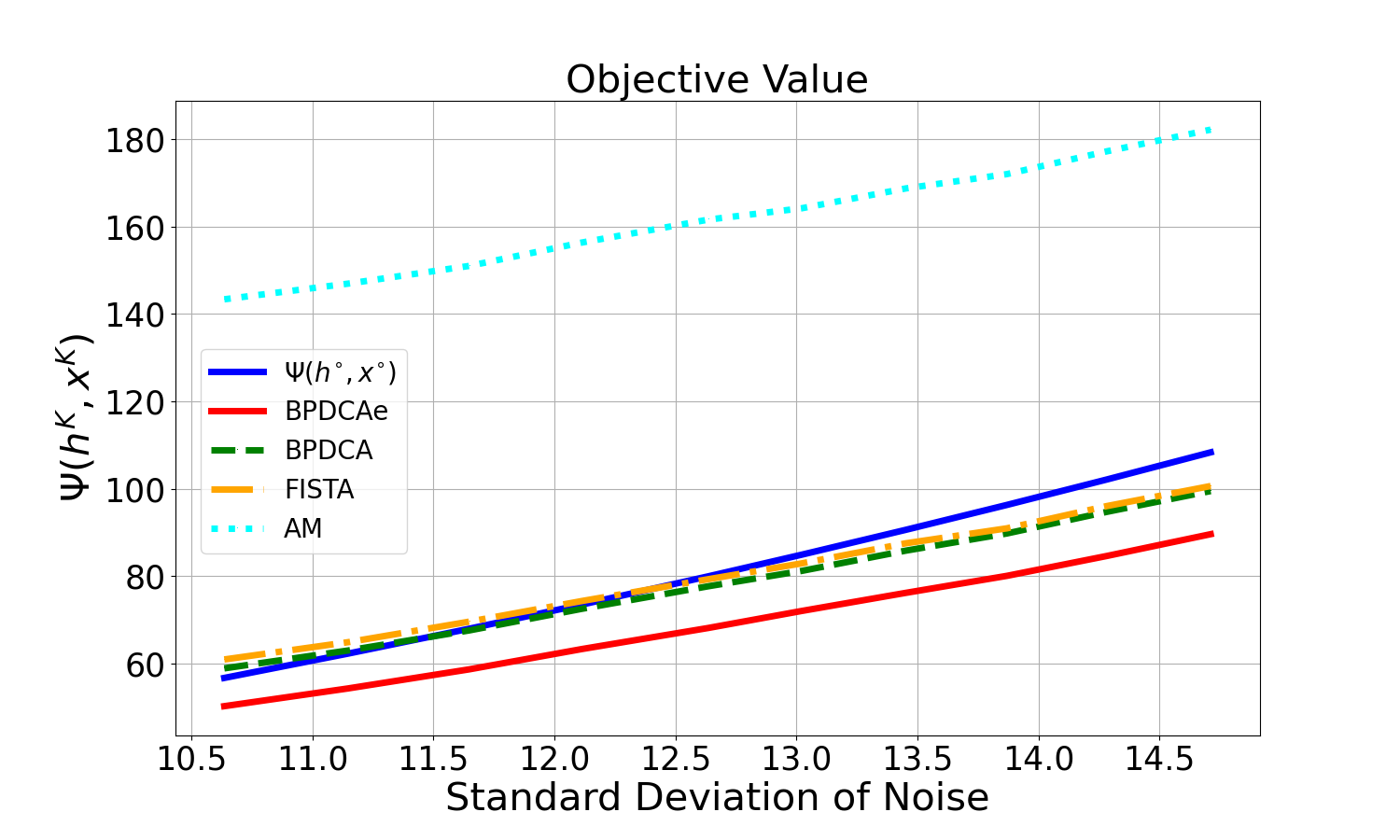}
    \subcaption{Plots of the objective value $\Psi(\bm{h}^K, \bm{x}^K)$.}
    \label{fig:objective-noise}
\end{minipage}
\hfill
\begin{minipage}[b]{0.49\linewidth}
    \centering
    \includegraphics[width=\textwidth]{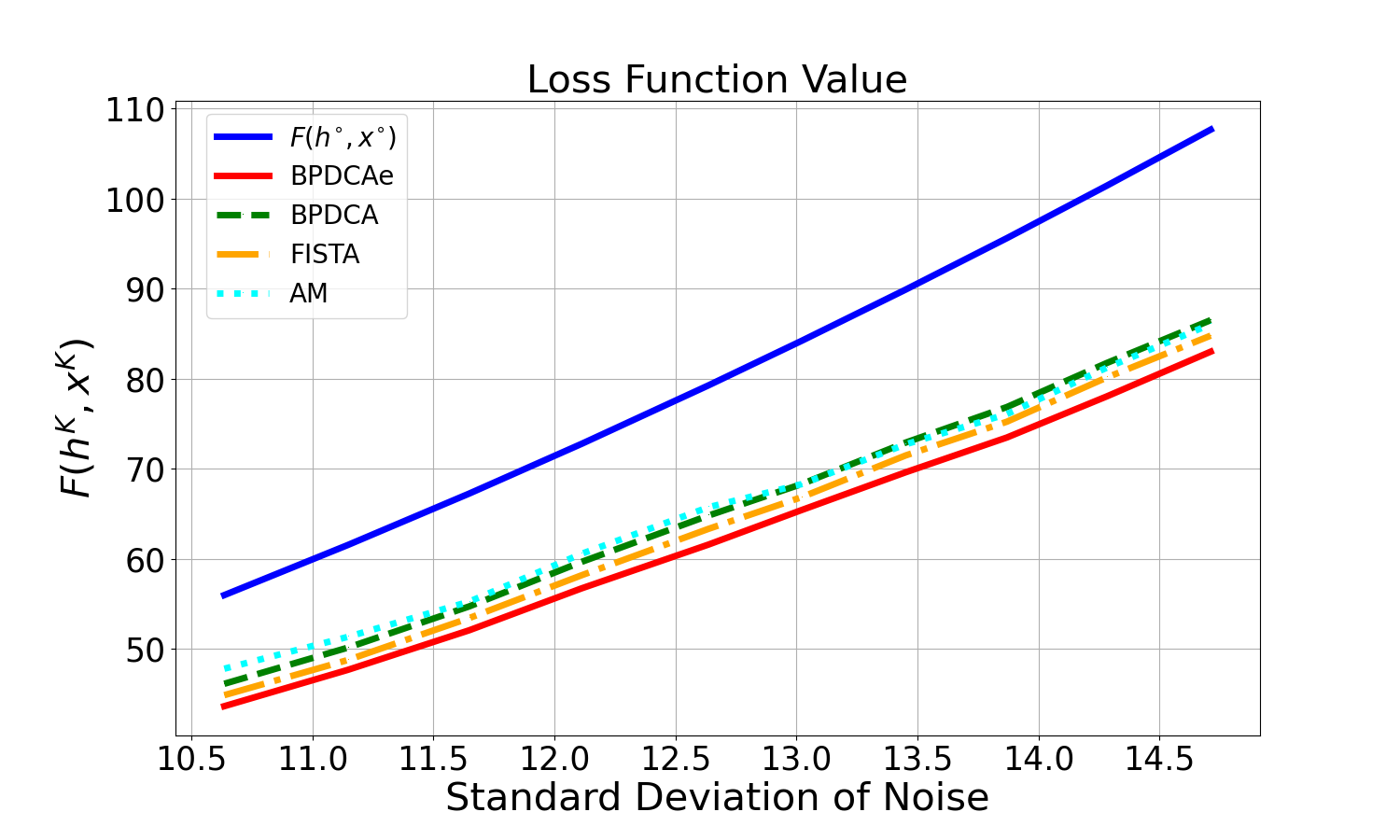}
    \subcaption{Plots of the loss function value $F(\bm{h}^K, \bm{x}^K)$.}
    \label{fig:loss-noise}
\end{minipage}\\
\begin{minipage}[b]{0.49\linewidth}
    \centering
    \includegraphics[width=\textwidth]{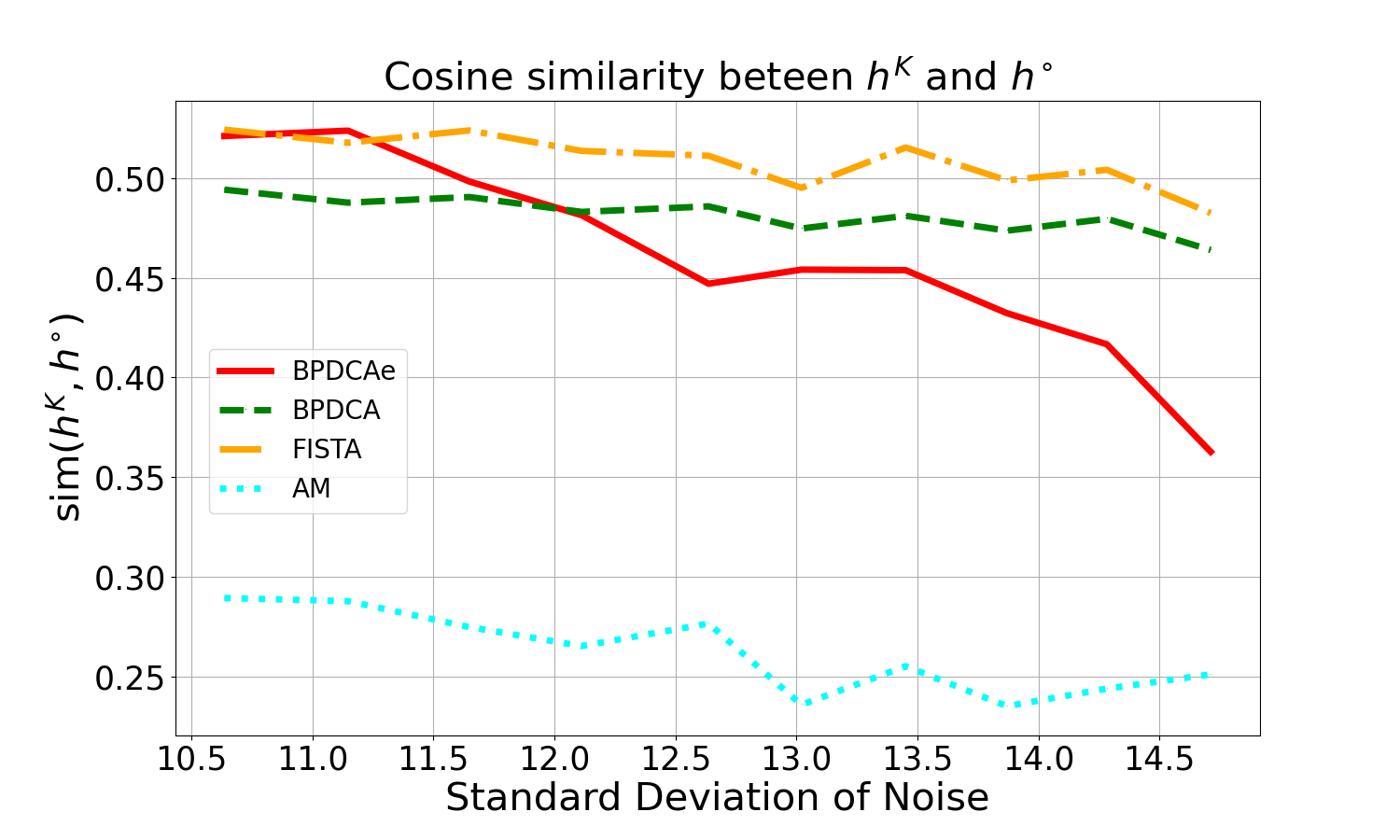}
    \subcaption{Plots of $\cossim (\bm{h}^K, \bm{h}^{\circ})$.}
    \label{fig:cos-sim-h-noise}
\end{minipage}
\hfill
\begin{minipage}[b]{0.49\linewidth}
    \centering
    \includegraphics[width=\textwidth]{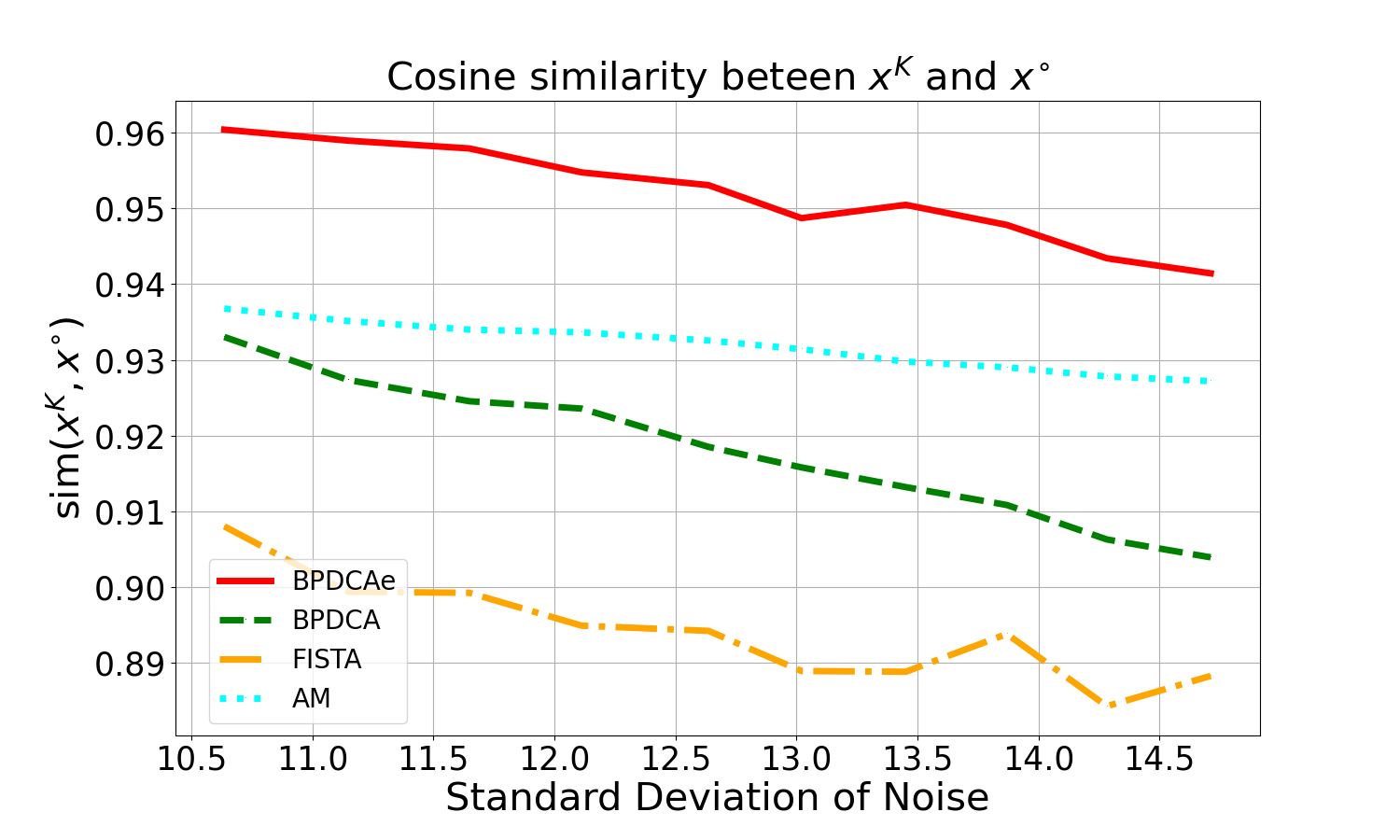}
    \subcaption{Plots of $\cossim (\bm{x}^K, \bm{x}^{\circ})$.}
    \label{fig:cos-sim-x-noise}
\end{minipage}
\caption{Plots of the objective value, the loss function value, and the cosine similarities recovered by each algorithm when $\tilde{\bm{y}}$ contains Poisson noise.}
\label{fig:result-noise}
\end{figure}

\begin{figure}
\begin{minipage}[b]{0.16\linewidth}
    \centering
    \includegraphics[width=\textwidth]{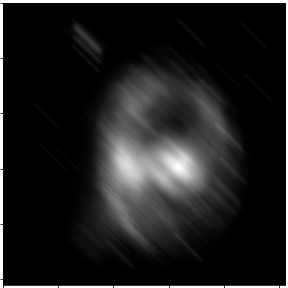}
    \subcaption{$\tilde{\bm{y}}$}
    \label{fig:noisy-image}
\end{minipage}
\hfill
\begin{minipage}[b]{0.16\linewidth}
    \centering
    \includegraphics[width=\textwidth]{recover-result-gt.png}
    \subcaption{$\bm{h}^{\circ}$ and $\bm{x}^{\circ}$}
    \label{fig:gt-noisy}
\end{minipage}
\hfill
\begin{minipage}[b]{0.16\linewidth}
    \centering
    \includegraphics[width=\textwidth]{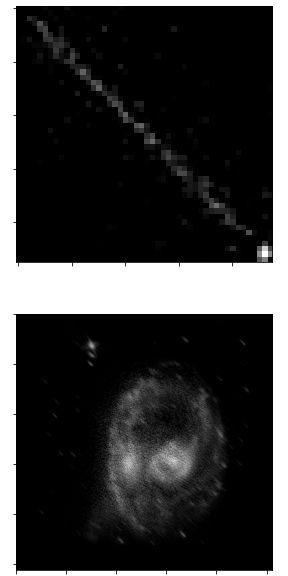}
    \subcaption{BPDCAe}
    \label{fig:bpdcae-noise}
\end{minipage}
\hfill
\begin{minipage}[b]{0.16\linewidth}
    \centering
    \includegraphics[width=\textwidth]{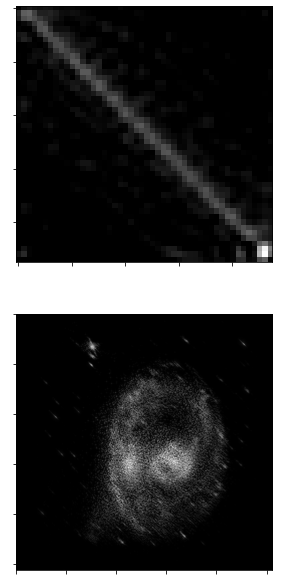}
    \subcaption{BPDCA}
    \label{fig:bpdca-noise}
\end{minipage}
\hfill
\begin{minipage}[b]{0.16\linewidth}
    \centering
    \includegraphics[width=\textwidth]{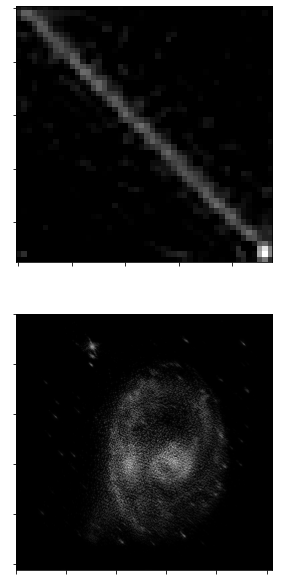}
    \subcaption{FISTA}
    \label{fig:fista-noise}
\end{minipage}
\hfill
\begin{minipage}[b]{0.16\linewidth}
    \centering
    \includegraphics[width=\textwidth]{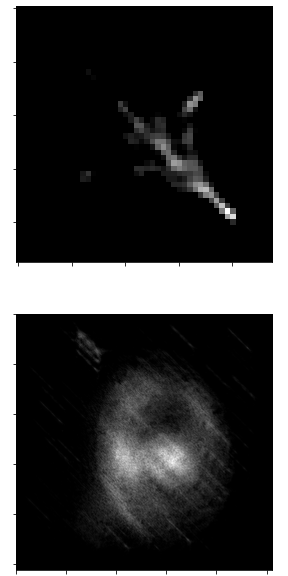}
    \subcaption{AM}
    \label{fig:am-noise}
\end{minipage}
\caption{(a) the noisy data; (b--f) the upper row shows $\bm{h}^K$, and the lower row shows $\tilde{\bm{A}}\bm{x}^K$ when $\tilde{\bm{y}}$ contains Poisson noise.}
\label{fig:compare-recover-noise}
\end{figure}

Finally, we demonstrate that BPDCA(e) is effective with another blur kernel.
We executed similar experiments using a Gaussian blur $\bm{f}$ and another image $\bm{g}$. Figures~\ref{fig:objective-log-diff-butterfly},~\ref{fig:cos-sim-h-butterfly}, and~\ref{fig:cos-sim-x-butterfly} show $\{\log_{10}|\Psi(\bm{h}^k,\bm{x}^k) - \Psi^{\circ}|\}$, $\{\cossim (\bm{h}^k, \bm{h}^{\circ})\}$, and $\{\cossim (\bm{x}^k, \bm{x}^{\circ})\}$, respectively, and Figure~\ref{fig:compare-recover-butterfly} shows the recovered images.
The performance in the sense of the objective values and the cosine similarities of BPDCAe is almost the same as that of FISTA.
The images recovered by BPDCAe, FISTA, and AM have almost no difference from Figures~\ref{fig:bpdcae-butterfly},~\ref{fig:fista-butterfly}, and~\ref{fig:am-butterfly}.
Thus, depending on types of the blur kernel $\bm{f}$ and the image $\bm{g}$, BPDCA(e) has the same results as the other algorithms.
As we saw here, the performance of BPDCA(e) is almost the same as or superior to that of the other algorithms.

\begin{figure}[t]
    \centering
    \includegraphics[width=.9\textwidth]{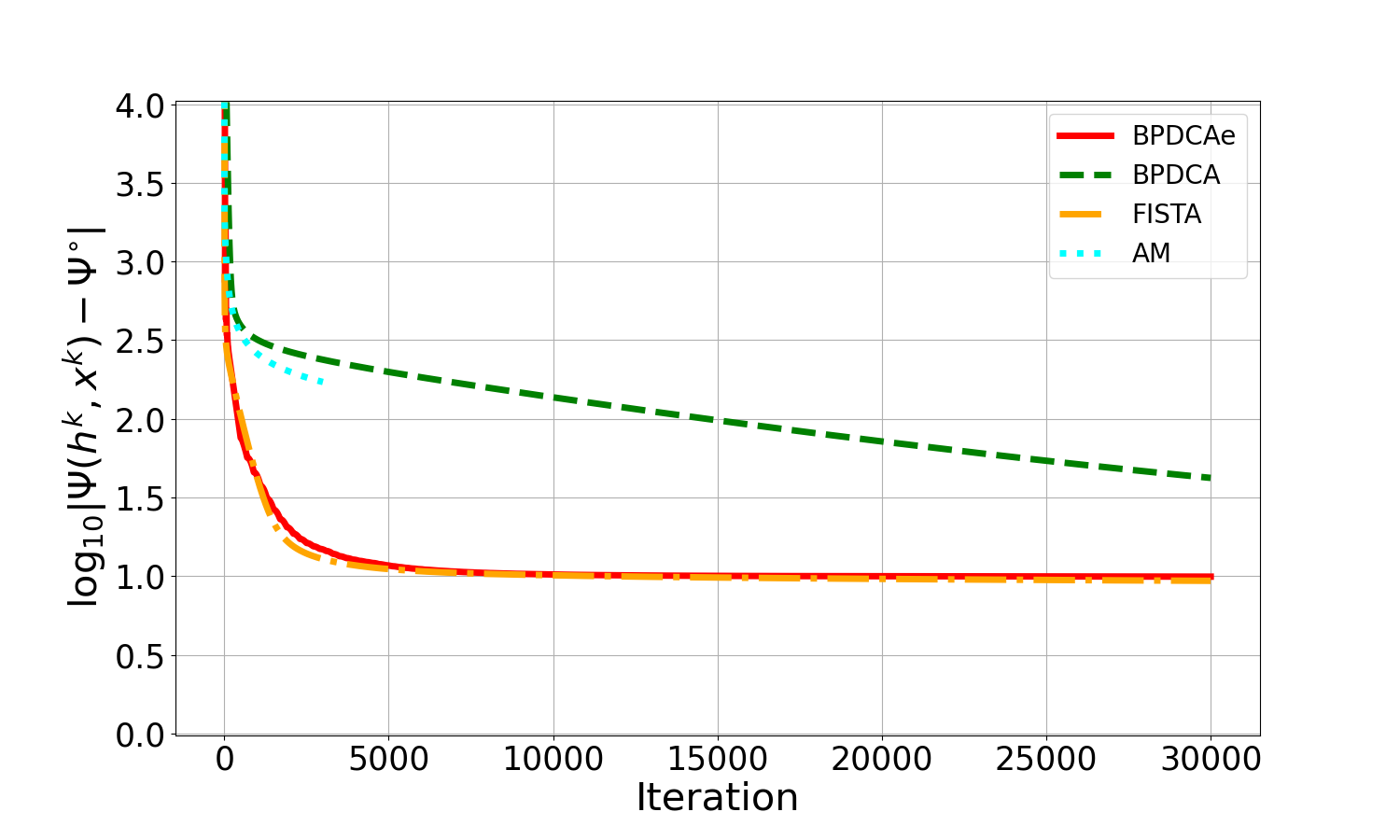}
    \caption{Plots of $\{\log_{10}|\Psi(\bm{h}^k, \bm{x}^k)-\Psi^{\circ}|\}$ at each iteration when $\bm{f}$ is a Gaussian blur.}
    \label{fig:objective-log-diff-butterfly}
\end{figure}

\begin{figure}[t]
\begin{minipage}[b]{0.49\linewidth}
    \centering
    \includegraphics[width=\textwidth]{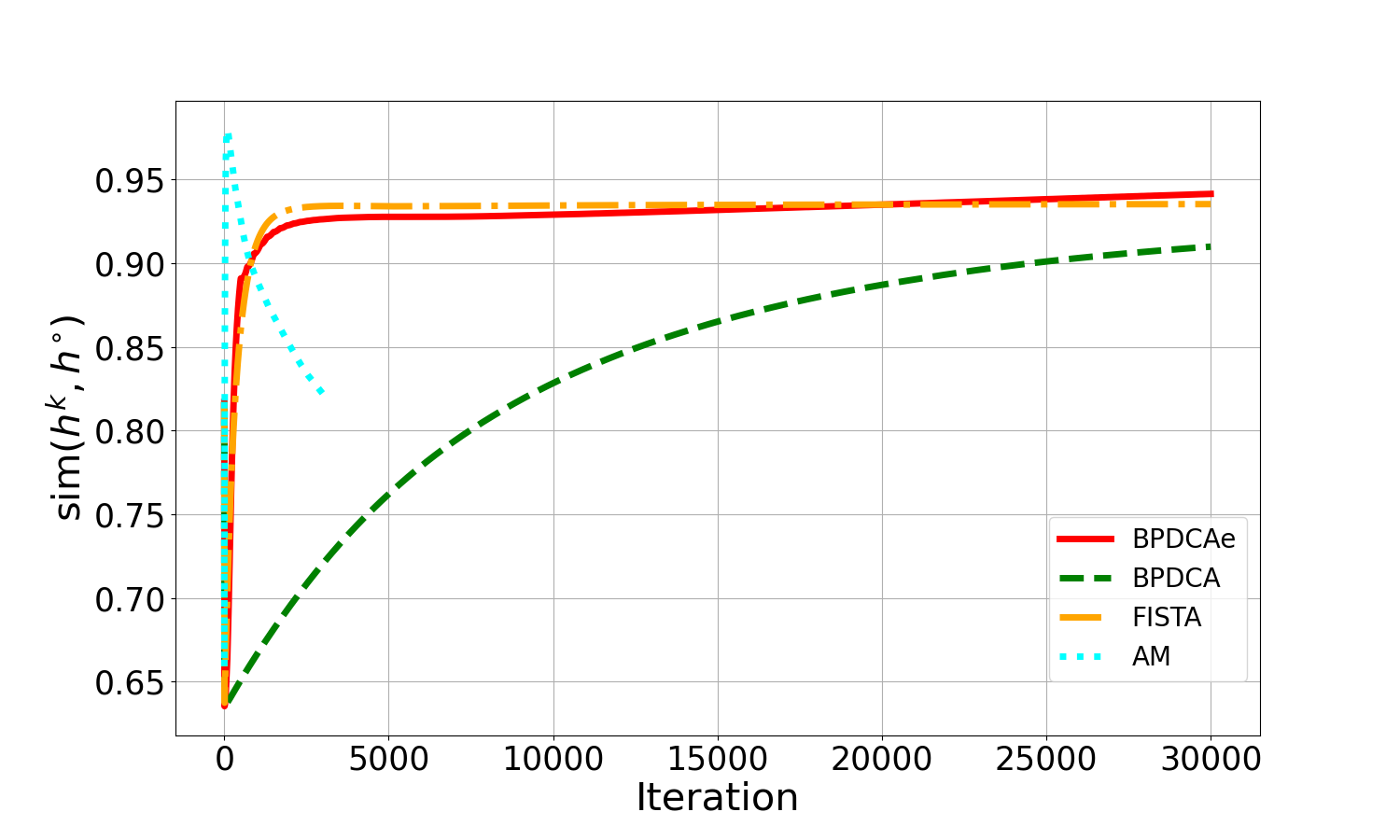}
    \subcaption{Plots of $\{\cossim (\bm{h}^k, \bm{h}^{\circ})\}$.}
    \label{fig:cos-sim-h-butterfly}
\end{minipage}
\hfill
\begin{minipage}[b]{0.49\linewidth}
    \centering
    \includegraphics[width=\textwidth]{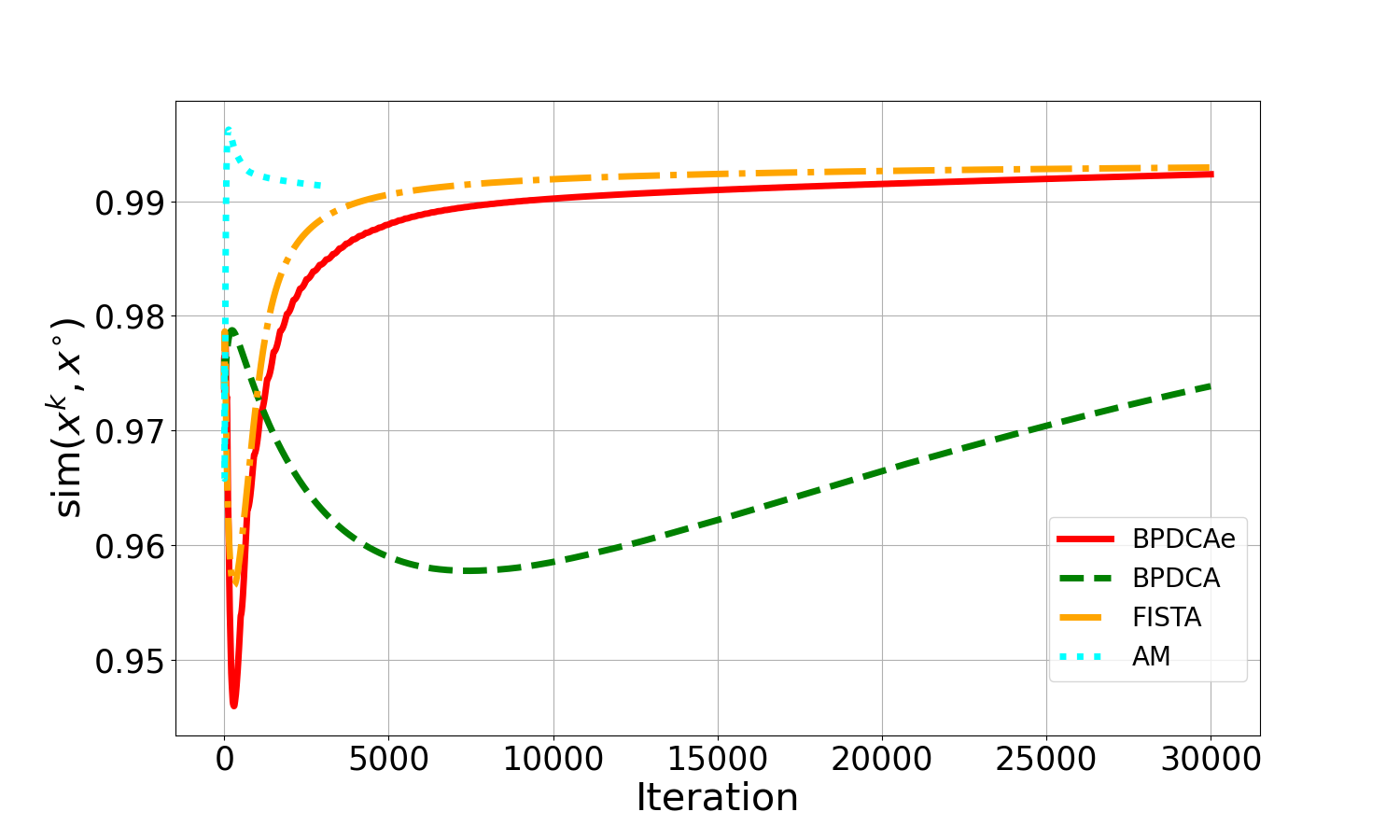}
    \subcaption{Plots of $\{\cossim (\bm{x}^k, \bm{x}^{\circ})\}$.}
    \label{fig:cos-sim-x-butterfly}
\end{minipage}
\caption{Plots of the cosine similarities between the $k$th point and the ground truth when $\bm{f}$ is a Gaussian blur.}
\label{fig:cos-sim-butterfly}
\end{figure}

\begin{figure}[t]
\begin{minipage}[b]{0.16\linewidth}
    \centering
    \includegraphics[width=\textwidth]{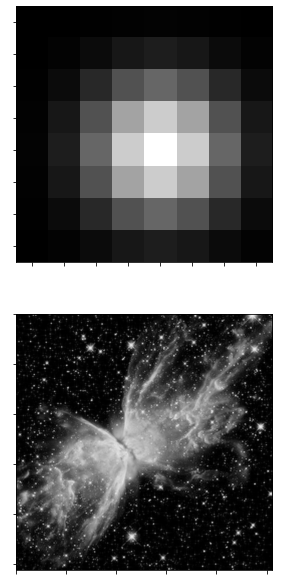}
    \subcaption{$\bm{h}^{\circ}$ and $\bm{x}^{\circ}$}
    \label{fig:butterfly}
\end{minipage}
\hfill
\begin{minipage}[b]{0.16\linewidth}
    \centering
    \includegraphics[width=\textwidth]{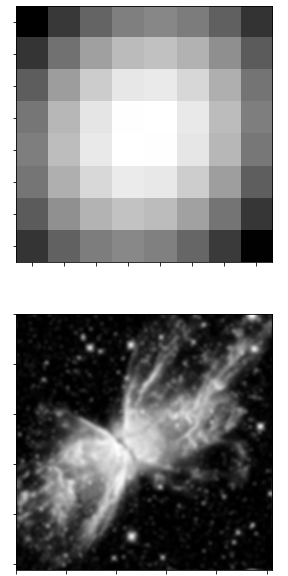}
    \subcaption{$\bm{h}^0$ and $\bm{x}^0$}
    \label{fig:init-butterfly}
\end{minipage}
\hfill
\begin{minipage}[b]{0.16\linewidth}
    \centering
    \includegraphics[width=\textwidth]{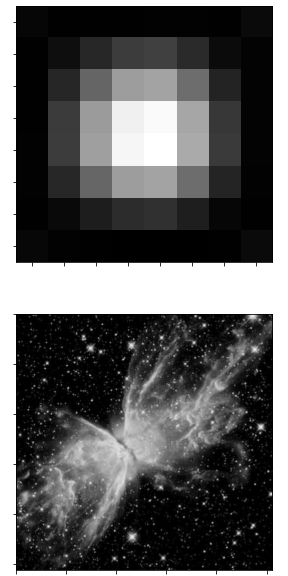}
    \subcaption{BPDCAe}
    \label{fig:bpdcae-butterfly}
\end{minipage}
\hfill
\begin{minipage}[b]{0.16\linewidth}
    \centering
    \includegraphics[width=\textwidth]{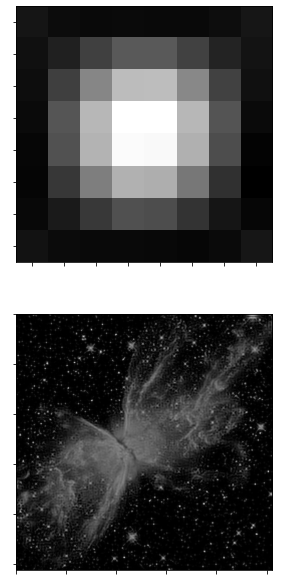}
    \subcaption{BPDCA}
    \label{fig:bpdca-butterfly}
\end{minipage}
\hfill
\begin{minipage}[b]{0.16\linewidth}
    \centering
    \includegraphics[width=\textwidth]{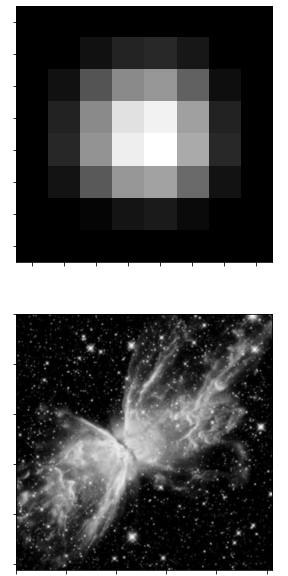}
    \subcaption{FISTA}
    \label{fig:fista-butterfly}
\end{minipage}
\hfill
\begin{minipage}[b]{0.16\linewidth}
    \centering
    \includegraphics[width=\textwidth]{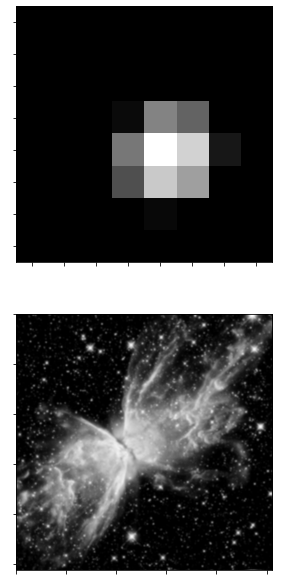}
    \subcaption{AM}
    \label{fig:am-butterfly}
\end{minipage}
\caption{The upper row shows $\bm{h}^K$, and the lower row shows $\tilde{\bm{A}}\bm{x}^K$ when $\bm{f}$ is a Gaussian blur.}
\label{fig:compare-recover-butterfly}
\end{figure}

\section{Conclusion}\label{sec:conclusion}
In blind deconvolution, recovering the original signal and the filter is naturally formulated as a minimization of a quartic objective function.
While existing first-order methods are not theoretically supported because the differentiable part of the objective function does not have a Lipschitz continuous gradient, our Bregman proximal DC algorithms efficiently work for this problem.
We found an appropriate DC decomposition and kernel generating distance, and we proved the convergence theoretically using them.
Our numerical experiments on image deblurring demonstrated that BPDCAe outperformed other existing algorithms and successfully recovered the original image.
The numerical successful of image deblurring is also because $g$ is represented sparsely in the wavelet domain.

While BPDCAe outperformed other existing algorithms, we still have several tasks.
In theory, although the convergence rate of BPDCA(e) characterized by the Kurdyka-\L ojasiewicz (KL) exponent has been established in~\cite{takahashi21}, the actual value of the KL exponent for blind deconvolution has not been revealed.
In practice, we have not discussed how to decide on better value of the regularization parameter $\theta$.

\section*{Acknowledgements}
This work was supported by JSPS KAKENHI Grant Number JP19K15247; JSPS KAKENHI Grant Number JP20H01951.

\section*{Declaration of Competing Interest}
The authors declare that they have no known competing financial interests or personal relationships that could have appeared to influence the work reported in this paper.

\bibliographystyle{plain}
\bibliography{main}

\begin{thebibliography}{10}

\bibitem{ahmed14}
A.~Ahmed, B.~Recht, and J.~Romberg.
\newblock Blind deconvolution using convex programming.
\newblock {\em IEEE Transactions on Information Theory}, 60(3):1711--1732,
  2014.

\bibitem{balzano07}
L.~Balzano and R.~Nowak.
\newblock Blind calibration of sensor networks.
\newblock In {\em Proceedings of the 6th International Conference on
  Information Processing in Sensor Networks}, pages 79--88, 2007.

\bibitem{beck09}
A.~Beck and M.~Teboulle.
\newblock A fast iterative shrinkage-thresholding algorithm for linear inverse
  problems.
\newblock {\em SIAM Journal on Imaging Sciences}, 2(1):183--202, 2009.

\bibitem{beck18}
J.~Bolte, S.~Sabach, M.~Teboulle, and Y.~Vaisbourd.
\newblock First order methods beyond convexity and {Lipschitz} gradient
  continuity with applications to quadratic inverse problems.
\newblock {\em SIAM Journal on Optimization}, 28(3):2131--2151, 2018.

\bibitem{bregman67}
L.~M. Bregman.
\newblock The relaxation method of finding the common point of convex sets and
  its application to the solution of problems in convex programming.
\newblock {\em USSR Computational Mathematics and Mathematical Physics},
  7(3):200--217, 1967.

\bibitem{chan00}
T.~F. Chan and C.~K. Wong.
\newblock Convergence of the alternating minimization algorithm for blind
  deconvolution.
\newblock {\em Linear Algebra and its Applications}, 316:259--285, 2000.

\bibitem{chen13}
J.~Chen, R.~Lin, H.~Wang, J.~Meng, H.~Zheng, and L.~Song.
\newblock Blind-deconvolution optical-resolution photoacoustic microscopy in
  vivo.
\newblock {\em Optics Express}, 21(6):7316--7327, 2013.

\bibitem{clarke90}
F.~H. Clarke.
\newblock {\em Optimization and nonsmooth analysis}.
\newblock SIAM, 1990.

\bibitem{fetick20}
R.~J.-L. Fétick, L.~M. Mugnier, T.~Fusco, and B.~Neichel.
\newblock Blind deconvolution in astronomy with adaptive optics: the parametric
  marginal approach.
\newblock {\em Monthly Notices of the Royal Astronomical Society},
  496(4):4209--4220, 2020.

\bibitem{jain22}
C.~Jain, A.~Kumar, A.~Chugh, and N.~Charaya.
\newblock Efficient image deblurring application using combination of blind
  deconvolution method and blur parameters estimation method.
\newblock {\em ECS Transactions}, 107(1):3695--3704, 2022.

\bibitem{jefferies93}
S.~M. Jefferies and J.~C. Christou.
\newblock Restoration of astronomical images by iterative blind deconvolution.
\newblock {\em The Astrophysical Journal}, 415:862--874, 1993.

\bibitem{krishnan11}
D.~Krishnan, T.~Tay, and R.~Fergus.
\newblock Blind deconvolution using a normalized sparsity measure.
\newblock In {\em Proceedings of the IEEE Conference on Computer Vision and
  Pattern Recognition 2011}, pages 233--240, 2011.

\bibitem{li19}
X.~Li, S.~Ling, T.~Strohmer, and K.~Wei.
\newblock Rapid, robust, and reliable blind deconvolution via nonconvex
  optimization.
\newblock {\em Applied and Computational Harmonic Analysis}, 47(3):893--934,
  2019.

\bibitem{liu95}
J.~S. Liu and R.~Chen.
\newblock Blind deconvolution via sequential imputations.
\newblock {\em Journal of the American Statistical Association},
  90(430):567--576, 1995.

\bibitem{nesterov18}
Y.~Nesterov.
\newblock {\em Lectures on convex optimization}, volume 137.
\newblock Springer, 2018.

\bibitem{perrone14}
D.~Perrone and P.~Favaro.
\newblock Total variation blind deconvolution: The devil is in the details.
\newblock In {\em Proceedings of the IEEE Conference on Computer Vision and
  Pattern Recognition 2014}, pages 2909--2916, 2014.

\bibitem{proakis06}
J.~G. Proakis and D.~G. Manolakis.
\newblock {\em Digital signal processing: principles algorithms and
  applications}.
\newblock Pearson, 2006.

\bibitem{takahashi21}
S.~Takahashi, M.~Fukuda, and M.~Tanaka.
\newblock New {Bregman} proximal type algorithms for solving {DC} optimization
  problems.
\newblock {\em arXiv preprint}, arXiv:2105.04873, 2021.

\bibitem{wu21}
Z.~Wu, C.~Li, M.~Li, and A.~Lim.
\newblock Inertial proximal gradient methods with bregman regularization for a
  class of nonconvex optimization problems.
\newblock {\em Journal of Global Optimization}, 79(3):617--644, 2021.

\bibitem{zhao16}
N.~Zhao, Q.~Wei, A.~Basarab, D.~Kouam{\'e}, and J.-Y. Tourneret.
\newblock Blind deconvolution of medical ultrasound images using a parametric
  model for the point spread function.
\newblock In {\em Proceedings of the IEEE International Ultrasonics Symposium
  2016}, pages 1--4, 2016.

\end{thebibliography}

\end{document}